\documentclass[11pt,letterpaper]{article}
\usepackage[lmargin=1.0in,rmargin=1.0in,bottom=1.0in,top=1.0in,twoside=False]{geometry}

\usepackage[utf8]{inputenc}
\usepackage[T1]{fontenc}
\usepackage{lmodern}

\usepackage{fullpage,amssymb,amsmath}
\usepackage{graphicx}
\usepackage{enumerate}
\usepackage{paralist}
\usepackage{tikz}
\usepackage{authblk}
\usetikzlibrary{shapes,patterns}

\usepackage{xcolor}
\usepackage{mathtools}
\usepackage{microtype}
\usepackage{amsfonts}
\usepackage{comment}
\usepackage[english]{babel}
\usepackage{mathrsfs}
\usepackage[ruled,vlined]{algorithm2e}
\usepackage{blindtext}
\usepackage{thmtools}

\usepackage{trimspaces}
\usepackage{nccfoots}
\usepackage{setspace}
\usepackage{inconsolata}
\usepackage{libertine}
\usepackage[absolute]{textpos}

\usepackage{enumerate}
\usepackage{todonotes}
\usepackage{longtable}

\definecolor{blue}{rgb}{0.1,0.2,0.5}
\definecolor{brown}{rgb}{0.6,0.6,0.2}
\usepackage[ocgcolorlinks, linkcolor={blue}, citecolor={brown}]{hyperref}
\usepackage[amsmath,amsthm,thmmarks,hyperref]{ntheorem}
\usepackage{enumerate}
\usepackage{latexsym}

\usepackage{cleveref}
\usepackage{stmaryrd}

\crefformat{page}{#2page~#1#3}%
\Crefformat{page}{#2Page~#1#3}%
\crefformat{equation}{#2(#1)#3}%
\Crefformat{equation}{#2(#1)#3}%
\crefformat{figure}{#2Figure~#1#3}%
\Crefformat{figure}{#2Figure~#1#3}%
\crefformat{section}{#2Section~#1#3}
\Crefformat{section}{#2Section~#1#3}
\crefformat{chapter}{#2Chapter~#1#3}
\Crefformat{chapter}{#2Chapter~#1#3}
\crefformat{chapter*}{#2Chapter~#1#3}
\Crefformat{chapter*}{#2Chapter~#1#3}
\crefformat{part}{#2Part~#1#3}
\Crefformat{part}{#2Part~#1#3}
\crefformat{enumi}{#2(#1)#3}
\Crefformat{enumi}{#2(#1)#3}
\crefformat{cthmin}{#2Theorem~#1#3}
\Crefformat{cthmin}{#2Theorem~#1#3}
\usepackage[mathlines]{lineno}

\newcommand*\patchAmsMathEnvironmentForLineno[1]{%
  \expandafter\let\csname old#1\expandafter\endcsname\csname #1\endcsname
  \expandafter\let\csname oldend#1\expandafter\endcsname\csname end#1\endcsname
  \renewenvironment{#1}%
     {\linenomath\csname old#1\endcsname}%
     {\csname oldend#1\endcsname\endlinenomath}}%
\newcommand*\patchBothAmsMathEnvironmentsForLineno[1]{%
  \patchAmsMathEnvironmentForLineno{#1}%
  \patchAmsMathEnvironmentForLineno{#1*}}%
\AtBeginDocument{%
  \patchBothAmsMathEnvironmentsForLineno{equation}%
  \patchBothAmsMathEnvironmentsForLineno{align}%
  \patchBothAmsMathEnvironmentsForLineno{flalign}%
  \patchBothAmsMathEnvironmentsForLineno{alignat}%
  \patchBothAmsMathEnvironmentsForLineno{gather}%
  \patchAmsMathEnvironmentForLineno{split}
  \patchBothAmsMathEnvironmentsForLineno{multline}}

\theoremnumbering{arabic}
\theoremstyle{plain}
\theoremsymbol{}
\theorembodyfont{\itshape}
\theoremheaderfont{\normalfont\bfseries}
\theoremseparator{.}

\newtheorem{theorem}{Theorem}
\crefformat{theorem}{#2Theorem~#1#3}
\Crefformat{theorem}{#2Theorem~#1#3}

\newcommand{\newtheoremwithcrefformat}[2]{%
  \newtheorem{#1}[theorem]{#2}%
  \crefformat{#1}{##2\MakeUppercase#1~##1##3}%
  \Crefformat{#1}{##2\MakeUppercase#1~##1##3}%
}
\newcommand{\newseptheoremwithcrefformat}[2]{%
  \newtheorem{#1}{#2}%
  \crefformat{#1}{##2\MakeUppercase#1~##1##3}%
  \Crefformat{#1}{##2\MakeUppercase#1~##1##3}%
}

\newtheoremwithcrefformat{lemma}{Lemma}

\newtheoremwithcrefformat{proposition}{Proposition}
\newtheoremwithcrefformat{observation}{Observation}
\newtheoremwithcrefformat{corollary}{Corollary}
\newseptheoremwithcrefformat{claim}{Claim}
\newtheorem*{problem*}{Problem}
\newtheorem*{conjecture*}{Conjecture}

\theoremstyle{definition}
\newtheorem*{example*}{Example}
\theorembodyfont{\upshape}
\newtheoremwithcrefformat{definition}{Definition}

\theoremstyle{nonumberplain}
\theoremheaderfont{\scshape}
\theorembodyfont{\normalfont}
\theoremsymbol{\ensuremath{\square}}

\theoremsymbol{\ensuremath{\lrcorner}}
\newenvironment{claimproof}{\noindent {\emph{Proof of Claim.}}}{\hfill$\lrcorner$\smallskip}

\newcommand{\eps}{\varepsilon}

\newcommand{\cC}{ \mathcal{C} }

\newcommand{\cL}{ \mathcal{L} }
\newcommand{\sL}{ \mathsf{L} }
\newcommand{\Var}{ \textsf{Var} }
\newcommand{\dist}{ \operatorname{dist}_k }
\newcommand{\oa}{ \overline{a} } 
\newcommand{\vphi}{\varphi} 
\newcommand{\oS}{\overrightarrow{S}}

\newcommand{\N}{\mathbb{N}}

\renewcommand{\epsilon}{\varepsilon}
\newcommand{\Oh}{\mathcal{O}}

\newcommand{\tx}{ \widetilde{x} }

\newcommand{\NEQ}{\mathsf{NEQ}}
\newcommand{\C}{\mathsf{Comp}}

\renewcommand{\leq}{\leqslant}
\renewcommand{\geq}{\geqslant}

\newcommand{\CSP}[1]{\textsc{CSP}(\ensuremath{#1})\xspace}

\newcommand{\listcoloring}[1]{\textsc{List \ensuremath{#1}-Coloring}\xspace}
\newcommand{\sat}[1]{\textsc{\ensuremath{#1}-CNF-Sat}\xspace}

\newcommand{\cc}{c^*}
\newcommand{\dd}{d^*}

\newcommand{\betterlb}{lower bound structure\xspace}

\newcommand{\containment}{\ensuremath{\mathsf{NP \subseteq coNP/poly}}\xspace}

\begin{document}
\title{Kernelization for list $H$-coloring\\for graphs
with small vertex cover}

\author[1,2]{Marta Piecyk\thanks{MP was funded by Polish National Science Centre, grant no. 2022/45/N/ST6/00237.}}
\author[ ]{Astrid Pieterse}
\author[1,3]{Pawe{\l} Rz\k{a}\.zewski}
\author[4]{Magnus Wahlstr\"om}
\affil[1]{Warsaw University of Technology}
\affil[2]{CISPA Helmholtz Center for Information Security}
\affil[3]{University of Warsaw}
\affil[4]{Royal Holloway University of London}
\date{}


\begin{titlepage}
\def\thepage{}
\thispagestyle{empty}
\maketitle
\begin{abstract}
  For two graphs $G$ and $H$, a \emph{homomorphism} from $G$ to $H$ is a mapping
  $\varphi \colon V(G) \to V(H)$ such that edges of $G$ map to edges of $H$.
  Additionally, given a list $\cL(v) \subseteq V(H)$ for every vertex $v \in V(G)$,
  a \emph{list homomorphism} from $(G,\cL)$ to $H$ is a homomorphism $\varphi$
  from $G$ to $H$ such that for every vertex $v \in V(G)$, it holds that $\varphi(v) \in \cL(v)$. 
  \textsc{List $H$-Coloring} is then the following problem: For a fixed graph $H$,
  given $(G,\cL)$ as input, is there a list homomorphism from $(G,\cL)$ to $H$? 
  Note that if $H$ is the complete graph on $q$ vertices, the problem is equivalent to \textsc{List $q$-Coloring}.

  We investigate the kernelization properties of \textsc{List $H$-Coloring} parameterized by the vertex cover number of $G$.
  That is, given an instance $(G,\cL)$
  and a vertex cover of $G$ of size $k$, can we reduce $(G,\cL)$ to an equivalent instance $(G',\cL')$ of \textsc{List $H$-Coloring}  where the size of $G'$ is bounded by a low-degree polynomial $p(k)$ in $k$?
  This question has been investigated previously by Jansen and Pieterse [Algorithmica 2019], who provided an upper bound, which turns out to be optimal if $H$ is a complete graph, i.e., for \textsc{List $q$-Coloring}.
  This result was one of the first surprising applications of the method of
  kernelization via bounded-degree polynomials.

  More generally, the kernelization of \textsc{List $H$-Coloring}
  parameterized by $k$ turns out to be equivalent to the kernelization
  of a related CSP (constraint satisfaction problem) on domain $V(H)$,
  parameterized by the number of variables; the limits of kernelization
  for CSPs under this parameter is an important but poorly understood question.

  We define two new integral graph invariants, denoted by $c^*(H)$ and $d^*(H)$, with $d^*(H)
  \leq c^*(H) \leq d^*(H)+1$, and show that the following hold.
  \begin{itemize}
  \item For every graph $H$, \textsc{List $H$-Coloring} has a kernel
    with $\mathcal{O}(k^{c^*(H)})$ vertices, hence size $\mathcal{O}(k^{c^*(H)} \log k)$,
    via a simple marking scheme.
  \item For every graph $H$, there is no kernel of size $\mathcal{O}(k^{d^*(H)-\varepsilon})$
    for any $\varepsilon > 0$, unless the polynomial hierarchy
    collapses.
  \item Furthermore, if $c^*(H) > d^*(H)$, then there is a kernel with
    $\mathcal{O}(k^{c^*(H)-\varepsilon})$ vertices where $\varepsilon \geq 2^{1-c^*(H)}$.
  \end{itemize}
  Additionally, we show that for some classes of graphs, including
  powers of cycles and graphs $H$ where $\Delta(H) \leq c^*(H)$ 
  (which in particular includes cliques), the
  bound $d^*(H)$ is tight, i.e., there is a kernel with $\mathcal{O}(k^{d^*(H)})$ vertices using
  the polynomial method. We conjecture that this holds in general.
\end{abstract}
\end{titlepage}
\section{Introduction}
For graphs $G$ and $H$, a \emph{homomorphism} from $G$ to $H$ is an edge-preserving mapping $\vphi: V(G) \to V(H)$, i.e., if $uv \in E(G)$, then $\vphi(u)\vphi(v)\in E(H)$.
If $G$ is given together with lists $\cL: V(G) \to 2^{V(H)}$ -- we will refer to them as \emph{$H$-lists} -- then a \emph{list homomorphism} $\vphi$ from $(G,\cL)$ to $H$ is a homomorphism from $G$ to $H$, which additionally respects lists, i.e., for every $v\in V(G)$ it holds that $\vphi(v)\in \cL(v)$.
To indicate that $\vphi$ is a list homomorphism from $G$ to $H$ respecting lists $\cL$, we will write $\vphi: (G,\cL)\to H$, and we write $(G,\cL)\to H$, to indicate that such a list homomorphism exists. For a fixed graph $H$, in the \listcoloring{H} problem, we are given a pair $(G,\cL)$, and we have to determine whether $(G,\cL)\to H$. We remark that in \listcoloring{H} it makes sense to allow graphs $H$ that have loops on vertices.

Note that if $H = K_q$, i.e., the complete graph on $q$ vertices, then \listcoloring{H} is equivalent to \listcoloring{q}.
Thus, \listcoloring{H} can be seen as a far-reaching generalization of \listcoloring{q}. For this reason, we will typically refer to the vertices of $H$ as \emph{colors}.

The complexity dichotomy for the non-list variant of the problem was shown in 1990 by Hell and Ne\v{s}et\v{r}il~\cite{HellN90hcol}: the problem is polynomial-time solvable if $H$ is bipartite or has a vertex with a loop, and \textsf{NP}-hard otherwise.
The dichotomy for the list version was shown by Feder, Hell, and Huang~\cite{DBLP:journals/jgt/FederHH03} (see also~\cite{DBLP:journals/jct/FederH98,DBLP:journals/combinatorica/FederHH99}). Here, the  \listcoloring{H} problem can be solved in polynomial time if $H$ is a so-called \emph{bi-arc graph} and otherwise it is \textsf{NP}-hard.

\paragraph{Kernelization of \listcoloring{H}.}
In this paper we are interested in kernelization  properties of the problem. Intuitively, we want to understand how much an instance of \listcoloring{H} can be compressed, in polynomial time, while preserving the answer.
Note that this problem makes sense only for graphs $H$ for which the problem is \textsf{NP}-hard, i.e., for non-bi-arc-graphs.
The following result of Chen, Jansen, Okrasa, Pieterse, Rzążewski~\cite{DBLP:journals/toct/ChenJOPR23} shows that in general no non-trivial compression is possible,
under standard complexity assumptions (i.e., unless \containment and the polynomial hierarchy collapses).

\begin{theorem}[\cite{DBLP:journals/toct/ChenJOPR23}]\label{thm:kernel-n}
Let $H$ be a non-bi-arc graph. Then \listcoloring{H} parameterized by the number of vertices of the input graph admits no generalized kernel of size $\Oh(n^{2-\eps})$, for any $\eps>0$, unless \containment.
\end{theorem}

Effectively, this results states that unless \containment,
an instance of \listcoloring{H} with $n$ vertices
cannot be compressed into $\mathcal{O}(n^{2-\eps})$ bits in polynomial time, regardless of encoding, 
without losing track of its yes/no-status. See \cref{sec:prel} for full definitions,
and~\cite{Book_kernelization_FLSZ19} for more on kernelization.

However, \cref{thm:kernel-n} does not exclude a possibility of non-trivial compression of instances with some restricted structure.
In particular, we are interested in instances with small vertex cover. Thus, from now on we assume that the input graph is given with a set $X$ of size at most $k$, such that $G - X$ has no edges.

It is easy to notice that such instances can be in polynomial time transformed into equivalent ones, where the number of vertices depends on $k$ only (such a compressed instance is called a \emph{kernel}).
Indeed, every vertex from $V(G) \setminus X$ can be characterized by (i) its neighborhood in $X$ and (ii) its list. Thus, there are at most $2^k \cdot 2^{|V(H)|} = \Oh(2^k)$ possible ''types'' of vertices in $V(G) \setminus X$, and it is sufficient to leave one vertex of each type (and the set $X$).
However, the number of vertices in such an instance is bounded by an exponential function of $k$. Can it be improved to a polynomial? As we discuss later, it turns out that this is indeed true. Furthermore, we are interested in bounding the degree of this polynomial. More precisely we study the the following problem.

\begin{problem*}
For every non-bi-arc graph $H$, determine the value $\delta = \delta(H)$ such that the  \listcoloring{H} problem, parameterized by the size $k$ of a vertex cover,
\begin{enumerate}
\item  admits a kernel with $\Oh(k^{\delta})$ vertices and edges,
\item does not have a kernel  of size $\Oh(k^{\delta-\varepsilon})$, for any $\varepsilon > 0$, unless \containment.
\end{enumerate}
\end{problem*}

This problem was previously studied by Jansen and Pieterse~\cite{JansenP19Coloring}, who showed that $\delta(H)$ (using the notation from the problem above) is upper-bounded by the maximum degree of a vertex in $H$, denoted by $\Delta(H)$.
(We remark that the result of Jansen and Pieterse~\cite{JansenP19Coloring} is stated for non-list variant of the problem, but the approach easily generalizes to the list variant).

\begin{theorem}[\cite{JansenP19Coloring}]
\label{thm:kernel-delta}
Then \listcoloring{H} parameterized by the size $k$ of a minimum vertex cover admits a kernel with $\Oh(k^{\Delta(H)})$ vertices and edges.
\end{theorem}

Furthermore, Jansen and Pieterse~\cite{JansenP19Coloring} proved that for \listcoloring{q}, i.e., for the case that $H$ is a complete graph, the bound given by \cref{thm:kernel-delta} is actually tight.


Their approach  uses a \emph{polynomial method}, which can be understood in two steps.
First, instances $(G,\cL)$ of \listcoloring{H} with a vertex cover $X$ are interpreted as 
CSPs, where $X$ is the variable set and vertices of $V(G) \setminus X$ act as
\emph{constraints} on the legal colorings of $X$; see below.
Next, these constraints are encoded into bounded-degree polynomials over $X$,
such that in order to preserve the solution space
it is enough to keep a basis of this set of polynomials.
Thus, an encoding of \listcoloring{H} into polynomials of degree $d$ implies a kernel with $\Oh(k^d)$ constraints for the CSP,
and a kernel with $\Oh(k^d)$ vertices and edges for \listcoloring{H}.
In this paper we continue this line of research, aiming to provide tight bounds of other graphs $H$, with possible loops on vertices.

\paragraph{Our contribution.}
Let $H$ be a graph that is not a bi-arc graph. 
We define two integer-valued quantities $\cc(H)$ and $\dd(H)$ (see \cref{sec:prel}).
We show that these values always differ by at most one -- more precisely, for every graph it holds that $\dd(H) \leq \cc(H) \leq \dd(H)+1$ (see \cref{lem:c-and-d}).
Interestingly, both possibilities are attainable.
The values $\cc(H)$ and $\dd(H)$ upper and lower bound the kernelization bounds for the problem, respectively, as follows.

\begin{restatable}{theorem}{thmmarking}
\label{thm:marking}
Let $H$ be a graph. Then \listcoloring{H} parameterized by the size $k$ of a vertex cover admits a kernel with $\Oh(k^{c^*(H)})$ vertices and edges.
\end{restatable}

\begin{restatable}{theorem}{thmlower}
\label{thm:lower-bound}
Let $H$ be a fixed non-bi-arc graph.
Then \listcoloring{H} parameterized by the size $k$ of a vertex cover does not have a kernel with $\Oh(k^{\dd(H)-\varepsilon})$ vertices and edges, for any $\varepsilon > 0$, unless \containment.
\end{restatable}

Thus, for graphs $H$ where $\cc(H)=\dd(H)$, the obtained bounds are tight. For graphs where $\dd(H) < \cc(H)$,  we conjecture that in fact, the lower bound is correct. 

\begin{conjecture*} \label{conjecture}
  For every graph $H$, \listcoloring{H} has a kernel with
  $\Oh(k^{\dd(H)})$ vertices, where $k$ is the size of a vertex cover of   the input graph $G$.
\end{conjecture*}

We provide some evidence in support of this. As informal evidence, for
every graph $H$ we have investigated via computer search, we found
that if $\dd(H)<\cc(H)$ then the problem allows an encoding as above
into polynomials of degree $\dd(H)$ (even over GF$(2)$).
More formally, we show that this holds for two classes of graphs:
powers of cycles (\cref{sec:powers-of-cycles}), and graphs where $\cc(H)$
is at least the max-degree of $H$ (\cref{sec:bounded-degree}).

Let us elaborate more on the latter class and recall that we can assume that $\dd(H) = \cc(H)-1$.
We observe that $\cc(H) \leq \Delta(H)+1$ (see \cref{lem:delta-c-plus1}).
If $\cc(H) = \Delta(H)+1$, i.e., $\dd(H) = \Delta(H)$, then the tight upper bound follows already from \cref{thm:kernel-delta}.
We also show that if $\cc(H)=\Delta(H)$, then $\dd(H)$ is the optimal exponent in the problem. 
Summarizing, we obtain the following result.

\begin{restatable}{theorem}{thmbddeg}
Let $H$ be a non-bi-arc graph such that $c^*(H) \geq \Delta(H)$. 
Then \listcoloring{H} parameterized by the size $k$ of the minimum vertex cover of the input graph admits a kernel with $\Oh(k^{d^*(H)})$ vertices and edges, but does not admit a kernel of size $\Oh(k^{d^*(H)-\varepsilon})$ vertices, for any $\varepsilon > 0$, unless \containment.
\end{restatable}

In addition, we show that for any graph $H$ with $\dd(H)<\cc(H)$,
there is small but positive value $\varepsilon=2^{1-\cc(H)}$ such that
\listcoloring{H} has a kernel with $\Oh(k^{\cc(H)-\varepsilon})$ vertices and edges. 

\begin{restatable}{theorem}{thmlhomeps}
\label{thm:lhom-eps}
Let $H$ be a non-bi-arc graph. Then for $\eps=2^{1-c^*(H)}$, the \listcoloring{H} problem parameterized by the size $k$ of a minimum vertex cover admits a kernel with $\Oh(k^{d^*(H)+1-\eps})$ vertices and edges.
\end{restatable}

Thus the simple bound $\Oh(k^{\cc(H)})$ is never tight unless $\cc(H)=\dd(H)$. 
This uses a connection to CSPs, which we survey next.

We note that if the degree bound of the smallest possible kernel were guaranteed to be an integer, then 
the conjecture would follow. However, for CSPs, problems with tight non-integral
kernelization degrees do exist~\cite{BrakensiekG25stoc} (and in fact,
exist for every rational power $p/q \geq 1$~\cite{Bartexponents}).
Thus, the question is whether this occurs also for the restricted class of CSPs arising from \listcoloring{H}.

\paragraph{The CSP connection}

One reason for studying \textsc{List $H$-Coloring} parameterized by the
vertex cover number is that it serves as a study case for
sparsification properties of CSPs parameterized by the number of variables.
Let us review the definitions. 
A \emph{constraint language} is a finite set of relations over some finite
domain $D$. A \emph{constraint} over $\Gamma$ is a tuple $(\mathbf{x}, R)$,
where $R \in \Gamma$ is a relation of some arity $r$ and $\mathbf{x}=(x_1,\ldots,x_r)$
is a tuple of variables, called the \emph{scope} of the constraint.
This is also written simply as $R(x_1,\ldots,x_r)$.
The constraint is \emph{satisfied} by an assignment $\varphi$ to the variables
if $(\varphi(x_1), \ldots, \varphi(x_r)) \in R$. 
For a fixed constraint language $\Gamma$, the problem $\CSP{\Gamma}$
is the constraint satisfaction problem over $\Gamma$, where the input
is a set of variables $X$ and a set $\cC$ of constraints over $\Gamma$, where
the scope of every constraint uses variables from $X$. The question is
whether there exists an assignment $\varphi \colon X \to D$ that
satisfies all constraints in $\cC$. The CSP framework is often used
as a setting for complexity characterizations; most famously, the
\emph{dichotomy theorem}, proven in 2017 after twenty years of study,
independently Bulatov~\cite{Bulatov17CSP} and Zhuk~\cite{Zhuk17CSP,Zhuk20CSP},
says that for every language $\Gamma$, $\CSP{\Gamma}$ is either in \textsf{P} or \textsf{NP}-complete.

The \textsc{$H$-Coloring} and \textsc{$H$-List Coloring} problems can
be seen as special cases of this framework.
Indeed, for a fixed graph $H$, we can interpret the edge set of $H$ as
a binary relation 
\[
  R_H = \{(u,v) \in V(H)^2 \mid uv \in E(H)\},
\]
in which case \textsc{$H$-Coloring} is equivalent to the CSP with language $\{R_H\}$,
where we simply enforce a constraint $R_H(u,v)$ for every edge $uv \in E(G)$,
and \textsc{List $H$-Coloring} is the CSP with the language that
additionally contains all unary constraints $L \subseteq 2^{V(H)})$.
From this perspective, graph homomorphism problems can serve as a
``trial ground'' where complex questions about the complexity of CSPs 
can be studied in a more well-behaved setting. For example, the aforementioned
complexity dichotomies of \textsc{$H$-Coloring}~\cite{HellN90hcol} and \textsc{List $H$-Coloring}~\cite{DBLP:journals/jgt/FederHH03}
by far precede the proofs of the CSP dichotomy theorem~\cite{Bulatov17CSP,Zhuk17CSP,Zhuk20CSP};
and the vertex/edge deletion variants of \listcoloring{H}, studied by
Chitnis et al.~\cite{ChitnisEM17vdel}, were an important milestone in
the study of FPT algorithms for CSP optimization problems such as
\textsc{MinCSP}~\cite{DBLP:conf/esa/OsipovPW24,DBLP:conf/soda/0002KPW23}.

The limits of kernelization of CSPs parameterized by the number $n$ of variables have seen
significant attention, and the polynomial method has been one of the
main tools employed~\cite{LagerkvistW20toct,DBLP:journals/algorithmica/ChenJP20,DBLP:conf/cp/Carbonnel22}. 
In fact, the first fine-grained lower bound on
kernel sizes was for \textsc{$q$-SAT}~\cite{DBLP:journals/jacm/DellM14},
and among the first applications of the polynomial method were a
non-trivial bound for \textsc{$q$-NAE-SAT}~\cite{DBLP:journals/toct/JansenP19},
which was the foundation for \cref{thm:kernel-delta}~\cite{JansenP19Coloring}.

More recently, two related notions have been studied --
\emph{non-redundancy} of CSPs, which is akin to a non-constructive kernelizability notion, 
and \emph{sparsification}, which is a stronger notion that preserves
not only the solution space of the CSP but also the approximate
portion of satisfied constraints~\cite{KhannaPS24,KhannaPS25} -- which, remarkably, have been shown
to coincide up to a polylogarithmic factor~\cite{BrakensiekG25stoc}.
Still, some very basic questions remain unanswered -- for example,
which Boolean languages admit kernelization to $\Oh(n)$ constraints, 
or have near-linear redundancy? 

Studying \listcoloring{H} via the CSP over the relation $R_H$ above is
not very informative, since \cref{thm:kernel-n} says that the
trivial kernel is optimal, but under the vertex cover parameter there
is another, richer connection to CSPs. Let $(G,\cL)$ be an instance of
\listcoloring{H} and $X$ a vertex cover of $G$. Then we can view $X$
as the variable set, and vertices $v \in V(G) \setminus X$ as
\emph{constraints} on the coloring used on $X$.
More precisely, let $\varphi \colon X \to V(H)$ be a coloring
of the vertex cover which is a list homomorphism from $(G[X],\cL)$ to $H$. 
Then $\varphi$ can be extended to a list homomorphism $(G,\cL) \to H$
if and only if, for every vertex $v \in V(G) \setminus X$, 
\[
  \exists a \in L(v) \; \forall x \in N_G(v) \colon a \in N_H(\varphi(x)),
\]
where $N_G(v)$ denotes the set of neighbors of $v$ in the graph $G$.
Since this depends only on the set of colors used for $N_G(v) \subseteq X$
in $\varphi$, we can ``project'' this to a constraint directly on $X$,
specifically a constraint
\[
  R_{L,r}(x_1,\ldots,x_r) \text{ where } R_{L,r} = \{(x_1,\ldots,x_r) \in V(H)^r \mid \exists
  a \in L \; \forall i \in [r] \ : \ a \in N_H(x_i)\}
\]
where $L=\cL(v)$ and $N_G(v)=\{x_1,\ldots,x_r\}$.
In this setting, the parameters $\cc(H)$ and $\dd(H)$ have natural CSP interpretations.
Namely, $\cc(H)$ is the \emph{decomposability} of $R_L$ -- that is,
for every value of $r=\deg_G(v)$ and every list $L$, the relation $R_{L,r}$ is equivalent to a
conjunction of constraints of arity $\cc(H)$. Thus, \listcoloring{H}
is equivalent to a CSP with $k$ variables over a language $\Gamma_H$
containing relations of arity up to $\cc(H)$.
The parameter $\dd(H)$ corresponds to a more technical, but established
notion: $q=\dd(H)+1$ is the smallest $q$ such that $\Gamma_H$ is preserved
by the so-called \emph{$q$-universal partial polymorphism} (see Lagerkvist
and Wahlström~\cite{LagerkvistW20toct}).
Within this framework, the $\varepsilon$-improved bound follows from the work of 
Carbonnel~\cite{DBLP:conf/cp/Carbonnel22}.

More strongly, for Boolean CSPs Chen, Jansen and Pieterse~\cite{DBLP:journals/algorithmica/ChenJP20}
showed that every $r$-ary relation except $r$-clauses (i.e., $r$-ary relations with only one excluded tuple)
can be captured by a set of polynomials of degree at most $r-1$.
Unfortunately, this is not true for general CSPs, and we cannot
exclude that the ``Booleanization'' of $R_{L,\cc(H)}$ yields an $r$-clause
even if $\dd(H) < \cc(H)$ for some graph $H$.

\section{Notation and preliminaries}
\label{sec:prel}

For a positive integer $n$, by $[n]$ we denote the set $\{1,\ldots,n\}$.
For a set $S$ and an integer $k$, by $\binom{S}{k}$ we denote the family of all $k$-element subsets of $S$.
The \emph{shadow} of a set $S$ to is
\[
  \delta S = \{S' \subset S \mid |S'|=|S|-1\}.
\]

\paragraph{Graphs.} 
Let $G$ be a graph with possible loops on vertices.
For a vertex $v\in V(G)$, by $N_G(v)$ we denote the set of neighbors of $v$ in $G$.
Note that $v \in N_G(v)$ if and only if $v$ has a loop.
The \emph{degree} of $v$ is $|N_G(v)|$ and is denoted by $\deg_G(v)$. Note that in this convention a loop contributes 1 to the degree of the vertex.
If $G$ is clear from the contex, then we will simply write $N(v)$ instead of $N_G(v)$, and $\deg(v)$ instead of $\deg_G(v)$.
We say that two vertices $u,v\in V(G)$ are \emph{incomparable} if $N(u)\not\subseteq N(v)$ and $N(v)\not\subseteq N(u)$. A set $L\subseteq V(G)$ is \emph{incomparable} if all its vertices are pairwise incomparable.

\paragraph{Homomorphisms.} 
Fix a graph $H$ and consider an instance $(G,\cL)$ of \listcoloring{H}. Suppose there is a vertex $v\in V(G)$ and two distinct vertices $x,y\in \cL(v)$ such that $N_H(x)\subseteq N_H(y)$, then the instance $(G,\cL')$ obtained by removing $x$ from $\cL(v)$ is equivalent to the instance $(G,\cL)$.
Indeed, if there is a list homomorphism $\vphi: (G,\cL)\to H$ with $\vphi(v)=x$, then we can also set $\vphi(v)=y$, as all neighbors of $x$ are also the neighbors of $y$.
Therefore, applying the above procedure exhaustively to an instance $(G,\cL)$, we can obtain in polynomial time an equivalent instance $(G,\cL')$ with all lists being incomparable sets.
We will call such an instance $(G,\cL')$ \emph{reduced}.

Sometimes we will abuse the notation and for an instance $(G,\cL)$ of \listcoloring{H} and for a subgraph $G'$ of $G$, we will write $(G',\cL)$ instead of $(G',\cL|_{V(G')})$.

Recall that \listcoloring{H} is polynomial-time solvable if $H$ is a bi-arc-graph, and \textsf{NP}-hard otherwise~\cite{DBLP:journals/jgt/FederHH03}.
The definition of bi-arc graphs is somewhat involved and not really relevant to our work, but let us list some properties of (non) bi-arc graphs that will be useful for us.

\begin{theorem}[\cite{DBLP:journals/jgt/FederHH03}]\label{thm:bi-arc}
Let $H$ be a simple graph.
\begin{enumerate}
\item If $H$ is non-bipartite or contains an induced cycle with at least $6$ vertices, then $H$ is non-bi-arc, and thus \listcoloring{H} is NP-hard.
\item If $H$ contains at most $2$ edges or is isomorphic to $C_4$, then $H$ is a bi-arc graph, and thus \listcoloring{H} is polynomial-time solvable.
\item If $H$ is a non-bi-arc graph, then there are vertices (not necessarily distinct) $v_1,v_2,v_3,v_4,v_5$ such that (i) $v_3$ is incomparable with $v_1$ and $v_5$, (ii) the vertices form a $v_1$-$v_5$ walk, i.e., for $i\in[4]$, we have $v_iv_{i+1}\in E(H)$, and (iii) $v_2v_5, v_1v_4\notin E(H)$.
\end{enumerate}
\end{theorem}


\paragraph{Kernels.} Let $\Sigma$ be a finite set (alphabet).
A \emph{parameterized problem} is a subset of $\Sigma^*\times \N$.
Let $g: \N\to\N$ be a computable function.
For parameterized problems $\mathcal{P},\mathcal{P'}\subseteq \Sigma^*\times \N$, \emph{a generalized kernel of $\mathcal{P}$ into $\mathcal{P'}$} is an algorithm $\mathcal{A}$ that given an instance $(x,k)\in \Sigma^*\times \N$, outputs in time polynomial in $|x|+k$ an instance $(x',k')$ such that:
\begin{enumerate}
\item $(x,k)\in \mathcal{P}$ if and only if $(x',k')\in \mathcal{P'}$,
\item $|x'|+k'\leq g(k)$.
\end{enumerate}
We say that $\mathcal{A}$ is \emph{a kernel} for $\mathcal{P}$ if $\mathcal{P}=\mathcal{P'}$.

\section{\boldmath Definition of $\cc(H)$ and $\dd(H)$}

Let $H$ be a graph with possible loops. 
Let $L,S \subseteq V(H)$. We say that $S$ \emph{has a common neighbor in} $L$ if there is a vertex in $L$ adjacent to every vertex in $S$.
Note that it might happen that a common neighbor belongs to $S \cap L$, it is only possible if such a vertex has a loop.

Let us define two graph invariants  that will be important in our paper.

\begin{definition}[$\cc(H)$]\label{def:c}
We define $\cc(H)$ as follows:
\[
\cc(H) := \max_{L\subseteq V(H)} \max \{ |S| \ | \ S\subseteq V(H) \text{ is a minimal set that does not have a common neighbor in } L \}.
\]
\end{definition}
In other words, $\cc(H)$ is the size of a largest possible set $S$, for which there is $L$, such that:
\begin{itemize}
\item $S$ has no common neighbor in $L$,
\item every proper subset of $S$ has a common neighbor in $L$.
\end{itemize}

\begin{definition}[$\dd(H)$]\label{def:d}
Let $H$ be a graph. A \emph{\betterlb} of order $d$ in $H$ is formed by a subset $L\subseteq V(H)$, distinct vertices $(x_1,\ldots,x_d)$, and non-necessarily distinct vertices $(x_1',\ldots,x_d')$, such that $x_i$ is incomparable with $x_i'$ for every $i \in [d]$, and
\begin{itemize}
\item $\bigcap_{i \in [d]} N(x_i) \cap L= \emptyset$, meaning that $x_1,\ldots,x_d$ do not have a common neighbor in $L$, but
\item $\bigcap_{i \in [d]} N(y_i) \cap L \neq \emptyset$ where $y_i \in \{x_i,x_i'\}$ and there exists $i \in [d]$ such that $y_i = x_i'$. This means that replacing at least one $x_i$ by its primed counterpart yields a set that does have a common neighbor in $L$.
\end{itemize}
By $\dd(H)$ we will denote the largest $d$ such that $H$ admits a \betterlb of order $d$.
\end{definition}

Let us point out that in our setting, in both definitions, the set $L$ will be a list of some vertex, and thus it would be natural to assume that $L$ is an incomparable set.
However, it turns out that such an assumption does not influence the definition.

\begin{observation}
Let $H$ be a graph. The following hold.
\begin{enumerate}
\item $
\cc(H) = \max_{\substack{L\subseteq V(H), \\ L \text{ is incomparable} }} \max \{ |S| \ | \ S\subseteq V(H) \text{ is a minimal set that does not have a common neighbor in } L \},
$

\item $d^*(H)$ is the maximum $d$ such that there is an incomparable set $L\subseteq V(H)$, distinct vertices $(x_1,\ldots,x_d)$, and non-necessarily distinct vertices $(x_1',\ldots,x_d')$ that form a lower bound structure of order $d$.
\end{enumerate}
\end{observation}
\begin{proof}
We will prove 1. -- the proof of 2. is analogous.
Let us define
\begin{equation*}
c'(H) := \max_{\substack{L\subseteq V(H), \\ L \text{ is incomparable} }} \max \{ |S| \ | \ S\subseteq V(H) \text{ is a minimal set that does not have a common neighbor in } L \}.
\end{equation*}
We aim to show that $c'(H)=c^*(H)$.
Clearly, it holds that $c'(H)\leq c^*(H)$.
So now let us show that $c'(H)\geq c^*(H)$.
Let $L\subseteq V(H)$, and let $S\subseteq V(H)$ be a minimal set that does not have a common neighbor in $L$ such that $c^*(H)=|S|$.
By the definition, if $L$ is incomparable, then $c'(H)\geq c^*(H)$, as desired.
So suppose that $L$ is not incomparable, and let $x,y\in L$ be such that $N_H(x)\subseteq N_H(y)$.
We will show that $S$ is also a minimal set that does not have a common neighbor in $L\setminus\{x\}$.
Indeed, since $S$ does not have a common neigbor in $L$, then $S$ does not have a common neighbor in $L\setminus \{x\}$.
Furthermore, for a proper subset $S'\subseteq S$, the set $S'$ has a common neighbor in $L$.
If the common neighbor of $S'$ in $L$ is $x$, then also $y$ is a common neighbor of $S'$, and thus $S'$ has a common neigbor in $L\setminus \{x\}$.

We can now exhaustively remove every vertex $x$ such that there is $y\in L$ with $N_H(x)\subseteq N_H(y)$, and obtain an incomparable set $L'\subseteq L$ such that $S$ is minimal set that does not have a common neighbor in $L'$.
This completes the proof.
\end{proof}

Recall that \listcoloring{H} is polynomial-time solvable for bi-arc graph, and \textsf{NP}-hard otherwise.
Therefore, in our work we mainly focus on non-bi-arc graphs.
The following observation shows that in such a case $d^*$ is always at least $2$.

\begin{observation}\label{obs:non-bi-arc-d}
Let $H$ be a non-bi-arc graph. Then $d^*(H)\geq 2$.
\end{observation}
\begin{proof}
Let $v_1,v_2,v_3,v_4,v_5$ be as in \cref{thm:bi-arc}~(3).
Let us define $x_1:=v_1$, $x_2:=v_5$, and $x'_1=x'_2=v_3$.
Moreover, we set $L:=\{v_2,v_4\}$.
Let us verify that $x_1,x_2,x'_1,x'_2,L$ form a lower bound structure of order $2$.
First, by \cref{thm:bi-arc}~(3), we have that $x_1=v_1$ is incomparable with $x'_1=v_3$, and $x_2=v_5$ is incomparable with $x'_2=v_3$.
Furthermore, the set $\{x_1,x_2\}=\{v_1,v_5\}$ does not have a common neighbor in $L=\{v_2,v_4\}$, as, by \cref{thm:bi-arc}~(3), $v_2v_5\notin E(H)$ and $v_1v_4\notin E(H)$.
So now consider the set $\{x''_1,x''_2\}$ such that for $i\in[2]$, we have $x''_i\in \{x_i,x'_i\}$, and for at least one $i\in [2]$, we have $x''_i=x'_i$.
Then the set $\{x''_1,x''_2\}$ contains $v_3$ and at most one of $v_1,v_5$.
If $\{x''_1,x''_2\}=\{v_3\}$, then the set clearly has a common neighbor in $L$.
If $\{x''_1,x''_2\}=\{v_1,v_3\}$, then the common neighbor of $\{x''_1,x''_2\}$ in $L$ is $v_2$, and if $\{x''_1,x''_2\}=\{v_3,v_5\}$, the the common neighbor of $\{x''_1,x''_2\}$ in $L$ is $v_4$.
Therefore, $x_1,x_2,x'_1,x'_2,L$ form a lower bound structure of order $2$, which completes the proof.
\end{proof}

We observe that the parameters introduced above, i.e., $\cc(H)$ and $\dd(H)$, may differ by at most one.

\begin{lemma}\label{lem:c-and-d}
For every graph $H$, it holds that $\cc(H)-1 \leq \dd(H) \leq \cc(H)$.
\end{lemma}

\begin{proof}
Let $c = \cc(H)$ and let $S=\{s_1,\ldots,s_c\}$ and $L$ be as in \cref{def:c}.
For all $i \in [c-1]$, define $x_i:=s_i$ and $x_i':=s_{c}$, and let $L':=L \setminus \bigcap_{i=1}^{c-1} N(s_i)$.
Clearly, the set $\{x_1,\ldots,x_{c-1}\}$ does not have a common neighbor in $L'$.
Since $S$ is a minimal set without a common neighbor in $L$, for every $i\in [c-1]$, the set $S\setminus \{s_i\}$ does have a common neighbor in $L$.
This common neighbor cannot be adjacent to $s_i$, and thus it cannot be in $\bigcap_{i=1}^{c-1} N(s_i)$.
Therefore, the set $S\setminus \{s_i\}$ does have a common neighbor in $L'$.
Moreover, by minimality of $S$, for every distinct $s_i,s_j\in S$ the set $N(s_i) \setminus N(s_j)$ intersects $L$ and thus in particular is nonempty. Therefore, for every $i\in[c-1]$, vertices $x_i$ and $x_i'$ are incomparable.
Summing up, the triple $L'$, $x_1,\ldots,x_{c-1}$, and $x_1',\ldots,x'_{c-1}$ is a lower bound structure of order $c-1=\cc(H)-1$ in $H$ and thus $\cc(H)-1\leq \dd(H)$.

Now let $d=\dd(H)$ and let $(L,(x_1,\ldots,x_d),(x_1',\ldots,x_d'))$ be lower bound structure of order $d$, as in \cref{def:d}.
Define $S:=\{x_1,\ldots,x_d\}$. Observe that $S$ does not have a common neighbor in $L$, but every proper subset of $S$ does.
Thus $S$ is an inclusion-wise minimal set, that does not have a common neighbor in $L$, and by the definition of $\cc$, it holds that $\dd(H)\leq \cc(H)$.
\end{proof}

By \cref{lem:c-and-d} we know that, for very graph $H$, 
we have either $c^*(H)=d^*(H)$ or $c^*(H) = d^*(H)+1$.
Interestingly, both possibilities are attainable, even for cycles.

\begin{observation}\label{obs:cycles-c-d}
Let $k\geq 5$. If $k\neq 6$, then $d^*(C_k)=c^*(C_k)=2$, and $c^*(C_6)=d^*(C_6)+1=3$.
\end{observation}
\begin{proof}
We will denote the consecutive vertices of $C_k$ by $0,1,\ldots,k-1$.
First let us consider the case that $k\neq 6$.
Let $S$ be a minimal set without a common neighbor in some set $L$.
If $|S|\geq 3$, then by the minimality of $S$, any distinct $u,v\in S$ have a common neighbor in $L$, so also in $V(H)$.
But it can be easily verified that it is only possible, if $k=6$ and the set $S$ contains precisely the vertices from one of the bipartition classes of $C_6$.
As we excluded $k=6$ in this case, we conclude that $c^*(C_k)\leq 2$. 

Now let us show that $d^*(C_k)\geq 2$.
We set $x_1:=0$, $x_2:=4$, $x_1':=x_2':=2$, and $L:=V(C_k)$. If $k=5$, then $0$ and $4$ are adjacent in $C_5$ and do not have a common neighbor.
If $k\geq 7$, then the distance between $0$ and $4$ in $C_k$ is at least $3$, so again $0$ and $4$ do not have a common neigbor in $C_k$.
After replacing at least one of $\{0,4\}$ with $2$, the set has a common neighbor, which is either $1$ or $3$.
Therefore $x_1,x_2,x_1',x_2',L$ form a lower bound structure of order $2$, and thus $d^*(C_k)\geq 2$.

So since now we consider the case $k=6$.
As discussed before, for any minimal set $S$ without a common neighbor in some $L$ such that $|S|\geq3$, it holds that any distinct $u,v\in S$, vertices $u$, $v$ have a common neighbor.
In $C_6$, there are two sets satisfying this property: $\{0,2,4\}$ and $\{1,3,5\}$.
It can be verified that both are minimal sets without a common neighbor in $L:=V(C_6)$, and thus $c^*(C_6)=3$. 

Now suppose that $d^*(C_6)=3$ and let $x_1,x_2,x_3,x'_1,x'_2,x'_3,L$ form a lower bound structure in $C_6$.
By symmetry of $C_6$, let $x_1=0$.
By the definition of a lower bound structure, the set $\{x_1,x_2,x'_3\}$ has a common neighbor in $L$, and thus in $V(C_6)$, so in particular $x_1$ and $x_2$ have a common neighbor.
Thus $x_2\in \{2,4\}$ (recall that the vertices $x_1,x_2,x_3$ have to be distinct, so $x_2\neq 0$).
Similarly, $x_1$ and $x_3$ have a common neigbor, and thus $x_3\in \{2,4\}$.
Since $x_2\neq x_3$, we have $\{x_1,x_2,x_3\}=\{0,2,4\}$ -- by symmetry, let us assume that $x_2=2$ and $x_3=4$.
Now it must hold that $x'_1\in\{2,4\}$ as $\{x'_1,x_2,x_3\}$ has a common neighbor, and, similarly, $x'_2\in \{0,4\}$ and $x'_3\in\{0,2\}$.
If the vertices $x'_1,x'_2,x'_3$ are all distinct, then the set $\{x'_1,x'_2,x'_3\}$ does not have a common neighbor in $V(C_6)$, and thus in $L$.
By symmetry, let $x'_1=x'_2$.
Then $x'_1\in \{2,4\}\cap\{0,4\}$, so $x'_1=x'_2=4$. 
Now if $x'_3=0$, then the set $\{x'_1,x_2,x'_3\}=\{4,2,0\}$ does not have a common neighbor -- a contradiction -- and if $x'_3=2$, then the set $\{x_1,x'_2,x'_3\}=\{0,4,2\}$ does not have a common neighbor -- a contradiction.
Therefore, $d^*(C_6)\leq 2$, which completes the proof.
\end{proof}

We note that $\cc(H)$ is bounded by $\Delta(H)+1$, and the equality is attained only for very specific graphs $H$.

\begin{lemma}\label{lem:delta-c-plus1}
For every graph $H$, it holds that $c^*(H) \leq \Delta(H)+1$.
\end{lemma}
\begin{proof}
By the definition of $c^*(H)$, there exists a set $S\subseteq V(H)$ with $|S|=c^*(H)-1$ such that $S$ has a common neighbor in $H$.
So $|S|\leq \Delta(H)$, and thus $c^*(H) \leq \Delta(H)+1$.
\end{proof}

Finally, let us argue that $\cc(H)$ can be arbitrarily smaller that $\Delta(H)$. Let $H$ be a graph obtained from an $r$-leaf star by subdividing each edge twice.
Such $H$ is not a bi-arc graph~\cite{DBLP:journals/combinatorica/FederHH99} and $\Delta(H) = r$.
On the other hand, a small case analysis shows that $\cc(H) = 2$.

\section{Lower bound}
In this section, we show the following lower bound.

\thmlower*

In order to prove \cref{thm:lower-bound}, we will need the following gadgets.

\begin{definition}[Inequality gadget $\NEQ(x,y)$]\label{def:neq}
Let $H$ be a graph and let $x,y$ be distinct vertices of $H$. An \emph{inequality gadget} $\NEQ(x,y)$ is a graph $F$ with $H$-lists $\cL$ and two designated  vertices $u,v$ such that:
\begin{enumerate}[({I}1.)]
\item $\cL(u)=\cL(v)=\{x,y\}$,
\item there exist list homomorphisms $\vphi, \psi: (F,L) \to H$ such that $\vphi(u)=\psi(v)=x$ and $\vphi(v)=\psi(u)=y$,
\item there is no list homomorphism $\vphi: (F,\cL) \to H$ such that $\vphi(u)=\vphi(v)$.
\end{enumerate}
\end{definition}

\begin{definition}[Compatibility gadget $\C(a,a',b,b')$]\label{def:compatib}
Let $H$ be a graph and let $a,a',b,b'$ be vertices of $H$. A \emph{compatibility gadget} $\C(a,a',b,b')$ is a graph $C$ with $H$-lists $\cL$ and two designated vertices $u,v$, such that:
\begin{enumerate}[(C1.)]
\item $\cL(u)=\{a,a'\}$ and $\cL(v)=\{b,b'\}$,
\item there exist list homomorphisms $\vphi,\psi: (C,\cL) \to H$ such that $\vphi(u)=a$, $\vphi(v)=b$, $\psi(u)=a'$, and $\psi(v)=b'$,
\item there is no list homomorphism $\vphi: (C,\cL) \to H$ such that $\vphi(u)=a$ and $\vphi(v)=b'$,
\item there is no list homomorphism $\vphi: (C,\cL) \to H$ such that $\vphi(u)=a'$ and $\vphi(v)=b$.
\end{enumerate}
\end{definition}

In the following lemma we show that both gadgets can be constructed.

\begin{lemma}\label{lem:gadgets}
Let $H$ be a graph. If $H$ has a lower bound structure $(L,\{x_1,\ldots,x_d\},\{x'_1,\ldots,x'_d\})$ of order $d\geq3$, then there exist:
\begin{itemize}
\item an inequality gadget $\NEQ(x_i,x_i')$ for every $i \in [d]$,
\item a compatibility gadget $\C(x_i,x'_i,x_j,x'_j)$ for every distinct $i,j\in [d]$.
\end{itemize}
Furthermore, each gadget has 10 vertices.
\end{lemma}

\begin{proof}
Without loss of generality let us assume that $i=1$ and $j=2$, the other cases clearly follow by the symmetry of the lower bound structure.

Consider any $\ell \in [d]$.
Recall that there exists a vertex $\tx'_\ell$ in $\left( \bigcap_{p \in [d]\setminus \{\ell\}} N(x_p) \cap N(x_\ell') \right) \cap L$; if there is more than one such a vertex, we choose an arbitrary one.
Note that $\tx_\ell' \notin N(x_\ell)$.
Furthermore, since $x_\ell$ and $x_\ell'$ are incomparable, there exists $\tx_\ell \in N(x_\ell)\setminus N(x_\ell')$.

\medskip
\noindent\textbf{Inequality gadget.}
First, let us construct an inequality gadget $\NEQ(x_1,x_1')$. We take two paths with lists of consecutive vertices:
\begin{itemize}
\item $\{x_1',x_1\}$, $\{\tx_1',\tx_2'\}$, $\{x_2,x_3\}$, $\{\tx_3',\tx_1'\}$, $\{x_1,x_1'\}$,
\item $\{x_1,x_1'\}$, $\{\tx_1,\tx_1'\}$, $\{x_1,x_2\}$, $\{\tx_2',\tx_3'\}$, $\{x_3,x_1\}$, $\{\tx_1',\tx_1\}$, $\{x_1',x_1\}$.
\end{itemize}
We join these two paths by identifying their endvertices, i.e., we identify the first vertices into a single vertex $u$ and the last vertices into a single vertex $v$. We claim that the constructed graph, with designated vertices $u$ and $v$, is an $\NEQ(x_1,x_1')$-gadget.
Clearly, property (I1.) of \cref{def:neq} is satisfied.

Now, if we map $u$ to $x_1$ and $v$ to $x_1'$, then we can extend this by mapping every vertex from the first path to the second vertex from the list and mapping every vertex from the second path to the first vertex from the list.
Similarly, if we map $u$ to $x_1'$ and $v$ to $x_1$, then we can extend this by mapping every vertex from the first path to the first vertex from the list and mapping every vertex from the second path to the second vertex from the list.
This proves that property (I2.) is satisfied.

Now note that mapping $u$ to $x_1$ forces every vertex from the first path to be mapped to the second vertex from its list, and thus $v$ cannot be mapped to $x_1$.
Similarly, mapping $u$ to $x_1'$ forces every vertex from the second path to be mapped to the second vertex from its list and thus $v$ cannot be mapped to $x_1'$.
This shows that property (I3.) is satisfied as well and the constructed graph is indeed an inequality gadget $\NEQ(x_1,x_1')$.

\medskip
\noindent\textbf{Compatibility gadget.} Now let us focus a compatibility gadget $\C(x_1,x_1',x_2,x_2')$.
Again, we introduce two paths, whose consecutive vertices have lists as follows:
\begin{itemize}
\item $\{x_1,x_1'\}$, $\{\tx_2',\tx_1'\}$, $\{x_3,x_2\}$, $\{\tx_1',\tx_3'\}$, $\{x_2,x_1\}$, $\{\tx_2,\tx_2'\}$, $\{x_2,x_2'\}$,
\item $\{x_1',x_1\}$, $\{\tx_1',\tx_1\}$, $\{x_3,x_1\}$, $\{\tx_2',\tx_3'\}$, $\{x_2',x_2\}$.
\end{itemize}
We again identify the first vertices of the paths into a vertex $u$ and the last vertices into $v$ and claim that the obtained graph satisfies the properties of a compatibility gadget.
Clearly, property (C1.) from \cref{def:compatib} is satisfied.

Now if we map $u$ to $x_1$ and $v$ to $x_2$, then we can extend this by mapping every vertex from the first path to the first vertex from its list and mapping every vertex from the second path to the second vertex from its list.
Similarly, if we map $u$ to $x_1'$ and $v$ to $x_2'$, then we can extend this by mapping every vertex from the first path to the second vertex from its list and every vertex from the second path to the first vertex from its list. This proves property (C2.).

Now note that mapping $u$ to $x_1$ forces every vertex from the first path to be mapped to the first vertex from its list and thus $v$ cannot be mapped to $x_2'$.
Similarly, mapping $u$ to $x_1'$ forces every vertex from the second path to be mapped to the first vertex from its list and thus $v$ cannot be mapped to $x_2$.
Hence, properties (C3.) and (C4.) are satisfied, and the constructed graph is indeed a compatibility gadget $\C(x_1,x_1',x_2,x_2')$.
\end{proof}

Now we are ready to prove \cref{thm:lower-bound}.
\begin{proof}[Proof of \cref{thm:lower-bound}]
If $\dd(H) \leq 2$, then result follows immediately from \cref{thm:kernel-n} -- note that the number $N$ of vertices satisfies $N \geq k$.

Thus let as assume that $d := \dd(H) \geq 3$. We will present a linear-parameter transformation from \sat{d}.
More precisely, for and instance $\Phi = C_1 \land C_2 \land \ldots \land C_m$ of \sat{d} with $n$ variables,
in polynomial time we will construct an instance $(G,\cL)$ of \listcoloring{H} such that:
\begin{enumerate}[(1.)]
\item $(G,\cL)$ is a yes-instance of \listcoloring{H} if and only if $\Phi$ is a yes-instance of \sat{d},
\item $G$ has a vertex cover of size $\Oh(n)$.
\end{enumerate}
As \sat{d} does not to have a generalized kernel of size $\Oh(n^{d-\eps})$, when parameterized by the number $n$ of variables~\cite{DBLP:journals/jacm/DellM14}, this will already imply our theorem.

Let $(L, (x_1,\ldots,x_d),(x_1',\ldots,x_d'))$ be a lower bound structure in $H$ of order $d$.

\paragraph{Variable gadget.}
For each variable $v$ in the formula $\Phi$, introduce a \emph{variable gadget} $\Var_v$ with $2d$ special vertices: $a_1,\ldots,a_d,\oa_1,\ldots,\oa_d$, such that:
\begin{enumerate}
\item $\cL(a_i)=\cL(\oa_i)=\{x_i,x_i'\}$ for every $i\in [d]$,
\item there exists a list homomorphism $\vphi_F: (\Var_v,\cL) \to H$ such that for every $i\in[d]$ it holds that $\vphi_F(a_i)=x_i$ and $\vphi_F(\oa_i)=x_i'$,
\item there exist a list homomorphism $\vphi_T: (\Var_v,\cL) \to H$ such that for every $i\in[d]$ it holds that $\vphi_T(a_i)=x'_i$ and $\vphi_T(\oa_i)=x_i$,
\item for any list homomorphism $\vphi: (\Var_v,\cL) \to H$, either for every $i\in [d]$ it holds that $\vphi(a_i)=x_i$ and $\vphi(\oa_i)=x_i'$,
or for every $i\in [d]$ it holds that $\vphi(a_i)=x_i'$ and $\vphi(\oa_i)=x_i$.
\end{enumerate}
We will interpret mapping all vertices $a_i$ to $x_i'$ as setting the variable $v$ to true,
and mapping all $a_i$ to $x_i$ as setting $v$ to false.

Now let us show how to construct $\Var_v$. First, introduce vertices $a_i,\oa_i$ with lists $\cL(a_i)=\cL(\oa_i)=\{x_i,x_i'\}$ for every $i\in[d]$. For every $i\in [d]$, we call \cref{lem:gadgets} to obtain an inequality gadget $\NEQ(x_i,x_i')$, and we identify the designated vertices of the gadget with $a_i$ and $\oa_i$, respectively.
Then, for every $i\in[d-1]$, we call \cref{lem:gadgets} to obtain a compatibility gadget $\C(x_i,x_i',x_{i+1},x'_{i+1})$, and identify its designated vertices $u,v$ with $a_i$ and $\oa_{i+1}$, respectively. This completes the construction of $\Var_v$ (see \cref{fig:var-gadget}).

\begin{figure}
\centering{
\begin{tikzpicture}[every node/.style={draw,circle,fill=white,inner sep=0pt,minimum size=8pt},every loop/.style={}]
\node[label=left:\footnotesize{$a_1$}] (a1) at (0,0) {};
\node[label=right:\footnotesize{$\oa_1$}] (b1) at (3,0) {};
\node[label=left:\footnotesize{$a_2$}] (a2) at (0,-1.5) {};
\node[label=right:\footnotesize{$\oa_2$}] (b2) at (3,-1.5) {};
\node[label=left:\footnotesize{$a_3$}] (a3) at (0,-3) {};
\node[label=right:\footnotesize{$\oa_3$}] (b3) at (3,-3) {};
\node[label=left:\footnotesize{$a_4$}] (a4) at (0,-4.5) {};
\node[label=right:\footnotesize{$\oa_4$}] (b4) at (3,-4.5) {};

\foreach \k in {1,2,3,4}
{
\draw (a\k)--++(1,0)--++(0,0.3)--++(1,0)--++(0,-0.6)--++(-1,0)--++(0,0.3);
\draw (2,1.5-1.5*\k)--(b\k);
\node[draw=none,fill=none] (n) at (1.5,1.5-1.5*\k) {$\NEQ$};
}

\foreach \k in {1,2,3}
{
\draw (a\k)--++(0,-0.3)--++(0.5,0)--++(0,-0.9)--++(-1,0)--++(0,0.9)--++(0.5,0);
\node[draw=none,fill=none] (c) at (0,0.75-1.5*\k) {$\C$};
}

\foreach \k in {2,3,4}
{
\draw (a\k)--++(0,0.3);
}

\end{tikzpicture}
}
\caption{The construction of a variable gadget $\Var_v$ for $d=4$.}\label{fig:var-gadget}
\end{figure}
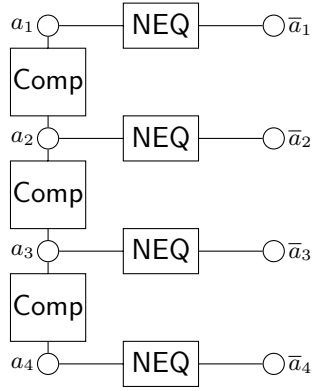

Let us verify that $\Var_v$ satisfies the desired properties.
Property 1. follows directly from the construction.
To see properties 2. and 3., define $\vphi_F(a_i):=x_i$, $\vphi_F(\oa_i):=x_i'$, $\vphi_T(a_i):=x_i'$, and $\vphi_T(\oa_i):=x_i$ for every $i\in [d]$.
By condition (I2.) of \cref{def:neq}, $\vphi_F$ and $\vphi_T$ can be extended to all vertices of inequality gadgets.
By condition (C2.) of \cref{def:compatib}, $\vphi_F$ and $\vphi_T$ can be extended to all vertices of compatibility gadgets, and thus $\Var_v$ satisfies properties 2. and 3.
Now assume that $\vphi: (\Var_v,\cL) \to H$ is a list homomorphism and $\vphi(a_i)=x_i$ for some $i\in[d]$.
By the properties (C3.) and (C4.) of \cref{def:compatib}, it holds that $\vphi(a_j)=x_j$ for every $j\in[d]$.
By property (I3.) of \cref{def:neq}, $\vphi(\oa_j)=x_j'$ for every $j\in [d]$.
Similarly, it can be verified that if $\vphi(a_i)=x_i'$ for any $i\in[d]$, then $\vphi(a_j)=x_j'$ and $\vphi(\oa_j)=x_j$ for every $j\in[d]$,
so the last property follows.

\paragraph{Construction of $(G,\cL)$.}
Now consider a clause $C =(\ell_1 \lor \ell_2 \lor \ldots \lor \ell_d)$, where each $\ell_i$ is a literal, i.e., a variable or its negation.
We assume that the ordering of literals within each clause is fixed.
For each clause $C$, we introduce a \emph{clause vertex} $c$ with list $\cL(c)=L$.
Now consider a literal $\ell_j$ of $C$, which corresponds to an occurrence of a variable $v$.
If this occurrence is positive, i.e., $\ell_j = v$, then we add an edge joining $c$ and the vertex $a_j$ of $\Var_v$.
Similarly, if this occurrence is negative, i.e., $\ell_j = \lnot v$, then we add an edge joining $c$ and the vertex $\oa_j$ of $\Var_v$.
This completes the construction of $(G,\cL)$.
Clearly $(G,\cL)$ was constructed in polynomial time.

Now observe that the clause vertices form an independent set, so the vertices of all variable gadgets $\Var_v$ form a vertex cover of $G$.
The number of the vertices in each variable gadget $\Var_v$ depends only on $d \leq |V(H)|$, so is a constant.
Thus there is a vertex cover in $G$ of size $\Oh(n)$.

\paragraph{Correctness.}
Now suppose that $\Phi$ is satisfiable and let $\psi$ be a satisfying assignment.
We will define a list $H$-coloring $\vphi : (G,\cL) \to H$.
For every variable $v$ which is set to true (resp., false) by $\psi$, we map all vertices of $\Var_v$ according to the mapping $\vphi_T$ (resp., $\vphi_F$).
It remains to extend $\vphi$ to clause vertices; note that they form an independent set.
Consider a clause $C=(\ell_1 \lor \ell_2 \lor \ldots \lor \ell_d)$ and its corresponding clause vertex $c$.
Let $N(c)=\{u_1,\ldots,u_d\}$ be the neighbors of $c$, such that $\cL(u_i)=\{x_i,x_i'\}$, i.e., $u_i$ is either $a_i$ or $\oa_i$ in the appropriate variable gadget, corresponding to $\ell_i$.
Since $\psi$ is a satisfying assignment, at least one literal $\ell_j$ of $C$ is set by $\psi$ to true.
This means that the vertex $u_j$ is mapped to $x_j'$.
By \cref{def:d}, the set $\{\vphi(u_1),\ldots,\vphi(u_d)\}$ has a common neighbor in $\cL(c)=L$, and thus $\vphi$ can be extended to $c$.

Suppose now that there exists a list homomorphism $\vphi:(G,\cL)\to H$.
For every variable $v$ of $\Phi$, if $\vphi(a_1)=x_1'$, where $a_1$ belongs to $\Var_v$, we set $v$ to true, and otherwise we set $v$ to false.
Let us show that it is a satisfying assignment of $\Phi$.
Consider a clause $C=(\ell_1 \lor \ell_2 \lor \ldots \lor \ell_d)$ and its corresponding clause vertex $c$.
Again let $N(c)=\{u_1,\ldots,u_d\}$ be the set of neighbors of $c$, such that $\cL(u_i)=\{x_i,x_i'\}$, i.e., $u_i$ belongs to $\Var_v$, where $v$ corresponds to $\ell_i$.
Suppose that $C$ is not satisfied, i.e., every literal in $C$ is false.
This means that for every $j \in [d]$, the vertex $u_j$ is mapped to $x_j$  by $\vphi$.
Therefore $\vphi(N(c))=\{x_1,\ldots,x_d\}$ and by \cref{def:d}, the set $\vphi(N(c))$ does not have a common neighbor in $\cL(c)=L$. Thus $\phi$ cannot be a list homomorphism from $(G,\cL)$ to $H$, a contradiction.
This completes the proof.
\end{proof}

\section{Simple kernel}
We first show that by a standard marking procedure, for an instance $(G,\cL)$ with vertex cover of size $k$, one can always obtain a kernel with $\Oh(k^{c^*(H)})$ vertices and edges. 
We will use the following observation.

\begin{observation}\label{obs:minimal-neighbors}
Let $H$ be a graph, let $(G,\cL)$ be an instance of \listcoloring{H}, and let $X$ be a vertex cover of $G$. Assume that a mapping $\vphi: X\to V(H)$ is a list homomorphism from $(G[X],\cL)$ to $H$, such that for every vertex $v\in V(G)\setminus X$, for every set $S\subseteq N_G(v)$ of size at most $c^*(H)$, it holds that $\bigcap_{u\in S} N_H(\vphi(u))\cap \cL(v)\neq\emptyset$. Then $\vphi$ can be extended to a list homomorphism from $(G,\cL)$ to $H$.
\end{observation}
\begin{proof}
Since $I=V(G)\setminus X$ is an independent set, it suffices to show that for every vertex $v\in I$, we can extend $\vphi$ to $v$.
This means that $\bigcap_{u\in N_G(v)} N_H(\vphi(u))\cap \cL(v)\neq \emptyset$.
For contradiction, suppose there is $v\in I$ such that $\bigcap_{u\in N_G(v)} N_H(\vphi(u))\cap \cL(v)=\emptyset$.
Moreover, let $S\subseteq N_G(v)$ be minimal set such that $\bigcap_{u\in S} N_H(\vphi(u))\cap \cL(v)=\emptyset$.
By minimality of $S$, the mapping $\vphi$ is injective on $S$ and $C=\{\vphi(u) \ | \ u\in S\}$ is a minimal set of vertices of $H$ without a common neighbor in $\cL(v)$.
Therefore, $|S|=|C|\leq c^*(H)$.
But then $\bigcap_{u\in S} N_H(\vphi(u))\cap \cL(v)\neq \emptyset$, a contradiction.
\end{proof}

Let us point out that in the following theorems, we assume that a vertex cover of size $k$ is given.
It is well-known that given a graph $G$ whose minimum vertex cover is of size $k$, in polynomial time a vertex cover of size at most $2k$ can be found.
Since we are interested in kernels of size in the form $\Oh(k^{f(H)})$, we can use a vertex cover of size at most $2k$, as then we have $\Oh((2k)^{f(H)})=\Oh(k^{f(H)})$ (we always treat $H$ as fixed, so $f(H)$ is some constant).
For simplicity, from now on we can just assume that the minimum vertex cover is given.

\thmmarking*

\begin{proof}
  We assume that $H$ is not a bi-arc graph, since otherwise the problem is tractable and there is a constant-sized kernel. Thus $\cc(H) \geq 2$ by \cref{obs:non-bi-arc-d}.
Let $(G,\cL)$ be an instance of \listcoloring{H} and let $X$ be the vertex cover of size $k$. Let $c:=c^*(H)$ and $h=|V(H)|$. We define the subgraph $G'$ of $G$ as follows. We add all vertices and edges of $G[X]$ to $G'$. Furthermore, for every set $X'\subseteq X$ of size at most $c$, and for every set $L\subseteq V(H)$, if there is at least one vertex $v\in V(G)\setminus X$ with list $L$ adjacent to all vertices of $X'$, then we fix one such $v$, and we add $v$ to $G'$ along with the edges from $v$ to $X'$. Let us point out that a vertex $v$ can be chosen for more than one set $X'$, and in such a case, this vertex is adjacent in $G'$ to all vertices of every such a set $X'$. That completes the construction of $G'$.

Clearly, $G'$ is constructed in polynomial time (recall that $H$ is a fixed graph, so $c$ and $h$ are constants). Since $G'$ is a subgraph of $G$, the size $k'$ of a minimum vertex cover in $G'$ is at most $k$. Furthermore, by the construction, the number of vertices of $G'$ is at most $k+k^{c}\cdot 2^{h}$. Indeed, there are $k$ vertices of $X$, and each vertex in $V(G')\setminus X$ was chosen for one of $k^{c}\cdot 2^{h}$ pairs $(X',L)$. Moreover, there are at most $k^2$ edges inside $X$, and for every pair $(X',L)$ we added at most $c$ edges. Therefore, there are at most $k^2+c\cdot 2^{h}\cdot k^{c}$ edges in $G'$. Since $H$ is fixed and by \cref{obs:non-bi-arc-d}, we conclude that $G'$ has $\Oh(k^{c^*(H)})$ vertices and edges.

\paragraph{Correctness.} Let us verify that $(G,\cL)$ is a yes-instance if and only if $(G',\cL)$ is a yes-instance.
Since $G'$ is a subgraph of $G$, if there exists a list homomorphism from $(G,\cL)$ to $H$, then the same mapping restricted to the vertices of $G'$ is a list homomorphism from $(G',\cL)$ to $H$.
Thus, it suffices to show that if $(G',\cL)$ is a yes-instance, then so is $(G,\cL)$.
So suppose there is a list homomorphism $\vphi': (G',\cL)\to H$.
We define $\vphi': (G',\cL)\to H$ as follows.
We first set $\vphi|_X=\vphi'$.
Note that $\vphi$ respects the edges inside $X$ since $\vphi'$ is a homomorphism.
Now we want to extend $\vphi$ to remaining vertices of $G$.
By \cref{obs:minimal-neighbors}, it suffices to show that for every $v\in V(G)\setminus X$, and for every set $S\subseteq N(v)$ of size at most $c^*(H)$, it holds that $\bigcap_{u\in S} \vphi(u)\cap \cL(v)\neq\emptyset$.
Suppose there is $v\in V(G)\setminus X$, and a set $S\subseteq N(v)$ of size at most $c^*(H)$such that $\bigcap_{u\in S} \vphi(u)\cap \cL(v)=\emptyset$.
Recall that at least one vertex $u$ with list $L=\cL(v)$ and adjacent to all vertices of $S$ was added to $G'$ along with its edges to $S$.
Since $\vphi'$ is a list homomorphism, $\vphi'(u)$ is a common neighbor of $\vphi'(S)=\vphi(S)$ and $\vphi'(u)\in \cL(u)=L=\cL(v)$, a contradiction.
This completes the proof.
\end{proof}

Combining \cref{thm:lower-bound}, \cref{thm:marking}, and \cref{lem:c-and-d} we obtain the following corollary.

\begin{corollary}\label{cor:d-lower-upper}
Let $H$ be a non-bi-arc graph. Then \listcoloring{H} parameterized by the size $k$ of a vertex cover:
\begin{enumerate}[(a)]
\item has a kernel with $\Oh(k^{d^*(H)+1})$ vertices and edges,
\item has no kernel with $\Oh(k^{d^*(H)-\eps})$ vertices and edges, for any $\eps>0$, unless \containment.
\end{enumerate}
\end{corollary}

\newcommand{\vect}[1]{\ensuremath{\mathbf{#1}}\xspace}
 
\section{Improved kernels}
Let us point out that for graphs $H$ such that $c^*(H)=d^*(H)$, the bounds given by \cref{thm:marking} and \cref{thm:lower-bound} are tight.
By \cref{lem:c-and-d}, for all the remaining graphs $H$, we have $c^*(H)=d^*(H)+1$.
Therefore, since now we focus on graphs $H$ with $c^*(H)=d^*(H)+1$.

First, we show that \listcoloring{H} parameterized by the size $k$ of the minimum vertex cover admits a kernel with $\Oh(k^{d^*(H)+1-\eps})$ vertices and edges. The result follows from~\cite[Theorem~22]{DBLP:conf/cp/Carbonnel22}. Let us state it here in a form that is suitable for us.

\begin{theorem}[Carbonnel~\cite{DBLP:conf/cp/Carbonnel22}]\label{thm:csp-epsilon}
Let $\Gamma$ be a constraint language over domain $D$ with arity $r\geq 2$. Assume that for every $r$-ary relation $R\in \Gamma$ the following holds:
\begin{itemize}
\item Let $c_1,\ldots,c_r,c'_1,\ldots,c'_r\in D$. If for every $(c''_1,\ldots,c''_r)\in \{c_1,c'_1\}\times\ldots,\times \{c_r,c'_r\}\setminus \{(c_1,\ldots,c_r)\}$, it holds that $(c''_1,\ldots,c''_r)\in R$, then $(c_1,\ldots,c_r)\in R$.
\end{itemize}
Then there exists a polynomial-time algorithm that takes an instance $(V,C)$ of $\CSP{\Gamma}$ with $n$ variables, and outputs an instance $(V,C')$ of $\CSP{\Gamma}$, where $C'\subseteq C$, with the same solution set as $(V,C)$, and with $\Oh(n^{r-\eps})$ constraints, where $\eps=2^{1-r}>0$.
\end{theorem}

\cref{thm:csp-epsilon} together with \cref{obs:non-bi-arc-d} yields the \cref{thm:lhom-eps} which we restate here.

\thmlhomeps*
\begin{proof}
Let $(G,\cL)$ be a reduced instance of \listcoloring{H} with a vertex cover $X$ of size $k$.
We first define a constraint language $\Gamma$ as follows.
Let $D=V(H)$, and for every $r\in [c^*(H)]$, for every $(r+1)$-tuple $\sL=(L_1,\ldots,L_r,L)$ of incomparable subsets of $V(H)$, we add to $\Gamma$ an $r$-ary relation $R_{\sL,r}$ that consists of all $r$-tuples from $L_1\times\ldots\times L_r$ that have a common neighbor in $L$.
We also add to $\Gamma$ a binary edge relation $E$, i.e., the set of all pairs $(u,v)$ such that $uv\in E(H)$, and for each incomparable set $L\subseteq V(H)$, we add $L$ as a unary relation.
This completes the definition of $\Gamma$.

\begin{claim}\label{claim:gamma-assump}
$\Gamma$ satisfies the assumptions of \cref{thm:csp-epsilon}.
\end{claim}
\begin{claimproof} 
For the sake of contradiction, suppose that $\Gamma$ does not satisfy the assumptions of \cref{thm:csp-epsilon}.
Then there is a relation $R_{\sL,c^*(H)}$ of arity $c^*(H)$ (note that the requirement in the statement of \cref{thm:csp-epsilon} is only for the relations of maximum arity), and there are $c_1,\ldots,c_r,c'_1,\ldots,c'_r\in V(H)$ such that for every $(c''_1,\ldots,c''_r)\in \{c_1,c'_1\}\times\ldots,\times \{c_r,c'_r\}\setminus \{(c_1,\ldots,c_r)\}$, it holds that $(c''_1,\ldots,c''_r)\in R$, but $(c_1,\ldots,c_r)\notin R$.
We set $x_i=c_i$ and $x'_i=c'_i$ for $i\in[c^*(H)]$.
Recall that the relation $R_{\sL,c^*(H)}$ was introduced for a $c^*(H)$-tuple $\sL=(L_1,\ldots,L_{c^*(H)},L)$ of incomparable subsets of $V(H)$ and $R\subseteq L_1\times\ldots\times L_{c^*(H)}$.
Thus, for every $i\in [c^*(H)]$, the vertices $x_i,x'_i$ are incomparable.
Moreover, by the definition of $x_i,x'_i$, we have that:
\begin{itemize}
\item $\bigcap_{i \in [d]} N(x_i) \cap L= \emptyset$, and
\item $\bigcap_{i \in [d]} N(y_i) \cap L \neq \emptyset$ where $y_i \in \{x_i,x_i'\}$ and there exists $i \in [d]$ such that $y_i = x_i'$. 
\end{itemize}
Therefore, $L$, $(x_1,\ldots,x_{c^*(H)})$, and $(x'_1,\ldots,x'_{c^*(H)})$ form a lower bound structure of size $c^*(H)=d^*(H)+1$ in $H$, a contradiction.
\end{claimproof}

\paragraph{$\CSP{\Gamma}$ instance.} Now let us define an instance $(V,C)$ of $\CSP{\Gamma}$ equivalent to $(G,\cL)$.
We first set $V=X$.
Furthermore, for every pair $(u,v)\in X^2$ such that $uv\in E(H)$, we add constraint $(E,(u,v))$.
Moreover, for every $v\in X$, we add the constraint $(\cL(v),v)$.
Finally, for every $r\in [c^*(H)]$, for every incomparable $L\subseteq V(H)$, and for every $r$-tuple $(v_1,\ldots,v_r)$ of vertices of $X$, if there is a vertex $u\in V(H)\setminus X$ with $\cL(u)=L$ and adjacent to all vertices of $(v_1,\ldots,v_r)$, then we add a constraint $(R_{\sL,r},(v_1,\ldots,v_r))$, where $\sL=(\cL(v_1),\ldots,\cL(v_r),L)$.
Note that for every $i\in [r]$, the set $\cL(v_i)$ is incomparable, and thus there is $R_{\sL,r}\in \Gamma$.
This completes the construction of the instance $(V,C)$.
It follows from \cref{obs:minimal-neighbors} and the definition of $(V,C)$ that there is one-to-one correspondence between solutions of the instance $(V,C)$ of $\CSP{\Gamma}$ and list homomorphisms $\vphi: (G[X],\cL)\to H$ that can be extended to $G$.

By \cref{claim:gamma-assump}, we can call \cref{thm:csp-epsilon} on instance $(V,C)$ and in polynomial time we obtain an instance $(V,C')$ with $C'\subseteq C$, the same solution set, and with $\Oh(k^{c^*(H)-\eps})$ constraints. 

\paragraph{\listcoloring{H} instance.} Let us now define the desired instance $(G',\cL)$ of \listcoloring{H}.
We start with $G[X]$.
Then, for every constraint of $C'$ which corresponds to a tuple $(v_1,\ldots,v_r)$ of vertices of $X$ and a set $L$, we fix one vertex $u$ with $\cL(u)=L$ and adjacent to all vertices of $(v_1,\ldots,v_r)$, and we add $u$ to $G'$ along with its all edges to $(v_1,\ldots,v_r)$.
This completes the construction of $G'$.

Observe that there are $k$ vertices and at most $k^2$ edges in $G'[X]$, and there are at most $\Oh(k^{c^*(H)-\eps})$ vertices in $V(G')\setminus X$ and at most $\Oh(k^{c^*(H)-\eps})$ edges that are not contained in $X$.

It remains to show the equivalence of instances $(G,\cL)$ and $(G',\cL)$.
Since $G'$ is a subgraph of $G$, if $(G,\cL)\to H$, then $(G',\cL)\to H$.
So now suppose that $(G',\cL)\to H$ and let $\vphi': (G',\cL) \to H$ be a list homomorphism.
We define $\vphi: (G,\cL) \to H$ as follows.
For every $v\in X$, we set $\vphi(v)=\vphi'(v)$.
Note that if we treat $\vphi$ as an assignment of the variables in $V$, then $\vphi$ satisfies all the constraints of $C'$.
By \cref{thm:csp-epsilon}, $\vphi$ satisfies all the constraints of $C$.
But this means that $\vphi$ can be extended to a list homomorphism from $(G,\cL)$ to $H$.
This completes the proof.
\end{proof}

\subsection{Polynomial method}
In this section we present how to obtain smaller kernel using low degree polynomials, similar as in~\cite{DBLP:journals/toct/JansenP19}.
Consider an instance $(G,\cL)$ of \listcoloring{H} with a vertex cover $X$ of size $k$.
For every $v\in X$, $u\in V(H)$, we introduce a boolean variable $y_{v,u}$ whose value will correspond to mapping $v$ to $u$ or not.

In a similar way to \cite[Definition 4]{JansenP19Coloring}, we will define a choice assignment.
\begin{definition}
Given sets $X$ and $V(H)$, let $y_{v,u} \in \{0,1\}$ for $v \in X$, $u \in V(H)$ be a set of boolean variables and let $\vect{y}$ be the vector containing all these variables. We say \vect{y} is given a \emph{choice assignment} if for all $v \in X$:
 \[\sum_{u\in V(H)} y_{v,u} =1.\]
\end{definition}
 
Intuitively, a vector \vect{y} is given a choice assignment if and only if it directly corresponds to mapping vertices of $X$ to $V(H)$ so that vertex $v$ is mapped to $u$ whenever $y_{v,u}=1$. 
By the definition of a choice assignment, this means each vertex is mapped to precisely one vertex.

For a mapping $\vphi: X\to V(H)$, we say that a choice assignment of \vect{y} \emph{corresponds to $\vphi$} if for every $v\in X$, we have $y_{v,u}=1$ if and only if $\vphi(v)=u$.

We now state what it means to be able to forbid a certain mapping by low degree polynomials.
\begin{definition}
Let $H$ be a graph, let $L_1,\ldots,L_r,L\subseteq V(H)$ be incomparable sets, and let $F=L_1 \times \ldots \times L_r$.
Let $S = (s_1,\ldots,s_r)\in F$ be a sequence such that $\bigcap_{i=1}^r N(s_i)\cap L = \emptyset$. 
We say that \emph{$S$ can be forbidden (on distinct vertices $v_1,\ldots,v_r\in X$) by polynomials of degree $d$ with respect to pair $(F,L)$} if there exists a polynomial $p_S$ of degree at most $d$ on variables $\vect{y} :=(y_{v,u} \mid v\in X,u\in V(H))$ such that for all choice assignments to $\vect{y}$ the following holds: 

For $\ell\in [r]$, let $s_\ell'\in V(H)$ be such that $y_{v_\ell,s'_\ell} = 1$ (note that there is exactly one such $s'_\ell$) and let $S'=(s'_1,\ldots,s'_r)\in F$.  Then:
\begin{itemize}
\item If $S' = S$, then $p(\vect{y}) \neq 0$.
\item If $S'$ has a common neighbor in $L$, meaning $\bigcap_{i=1}^r N(s'_i) \neq \emptyset$, then $p(\vect{y}) = 0$. 
\end{itemize}
\end{definition}
Note that the two cases in the above definition do not cover all possibilities. There are no requirements on what the polynomial does when $\vect{y}$ is not given a choice assignment, and there are also no requirements on what happens if $\vect{y}$ is given an assignment that corresponds to a set with no common neighbor in $L$, not equivalent to $S$. This is on purpose.

We will use the following result of Jansen and Pieterse~\cite{DBLP:journals/toct/JansenP19}.
\begin{theorem}[{{\cite[Theorem~3.1, Claim~3.2]{DBLP:journals/toct/JansenP19}}}]\label{thm:d-poly-root-csp}
Let $P$ be a set of equalities over $m$ variables $\textrm{Var}=\{x_1,\ldots,x_m\}$,
where each equality is of the form $f(x_1,\ldots,x_m)=0$ for some multivariate polynomial $f$ of degree $d$ over $\mathbb{Z}_2$.
In time polynomial in $|P|$ and $m$ we can find a subset $P' \subseteq P$ of size at most $m^d+1$, such that 
an assignment of variables in $\textrm{Var}$ satisfies all equalities in $P'$ if and only if it satisfies all equalities in $P$.
\end{theorem}
 
Now we are ready to show how to use low-degree polynomials to obtain smaller kernels.
The proof is similar to the one in~\cite{JansenP19Coloring}.
 
 \begin{theorem}\label{thm:kernelization-by-polynomials}
Let $H$ be a graph such that for every $r\leq c^*(H)$, for every $F=L_1\times \ldots \times L_r$, where $L_i\subseteq V(H)$ is incomparable for $i\in [r]$, and for every $L\subseteq V(H)$, every $S\in F$ which does not have a common neighbor in $L$, can be forbidden by polynomials of degree $d$ with respect to $(F,L)$. Then \listcoloring{H} parameterized by the size $k$ of a minimum vertex cover has a kernel with $\Oh(k^d + k^2)$ vertices and edges.
 \end{theorem}
  
\begin{proof}
Let $(G,\cL)$ be an instance of \listcoloring{H} and let $X$ be a minimum vertex cover in $G$ of size $k$. 
Clearly, $G[X]$ has at most $\Oh(|X|^2) = \Oh(k^2)$ vertices and edges.
Note that $I := V(G) \setminus X$ is an independent set in $G$.

We aim to construct an equivalent instance $(G', \cL)$ of \listcoloring{H}, where $G'$ is a subgraph of $G$.
The vertex set of $G'$ is $X \cup I'$ for some $I' \subseteq I$, and the edge sets consists of all edges with both endpoints in $X$,
and a subset of edges joining $I'$ with $X$. 
The set $I'$, as well as the edges incident to this set, will be selected by finding redundant vertices and edges.
To achieve this, we will represent the constraints on the coloring of the neighborhood of such vertices, by low-degree polynomials.

%

\paragraph{Variables and polynomials.} Let $x_1,\ldots,x_h$ be the vertices of $H$ and let $v_1,\ldots,v_k$ be the vertices of $X$.
For every $i \in [k]$, we introduce $h$ variables $y_{v_i,x_j}$ that together represent the color of $v_i$, i.e., $y_{v_i,x_j}$ corresponds to mapping $v_i$ to $x_j$.
Let $\vect{y}$ be the vector containing all variables $y_{v_i,x_j}$. 

So now let us construct a set $P$ of polynomial equalities over $\mathbb{Z}_2$.
Intialize $P$ as empty set.
For every $i \in [k]$, and for every $x_j \in V(H)\setminus \cL(v_i)$, we add to $P$ an equality $y_{v_i,x_j}=0$.
Later we will refer to these polynomials as \emph{list polynomials}.
Then, for each vertex $v \in V(G)\setminus X$, consider each sequence $Y= (u_1,\ldots,u_r)$ of distinct vertices in $N(v)$ with $r\leq c^*(H)$.
For every $S = (s_1,\ldots,s_r)$ with $s_i \in \mathcal{L}(v_i)$ and $\bigcap_{s\in S} N_H(s)\cap \cL(v) = \emptyset$, construct a polynomial $p_{v,S}$ of degree at most $d$ that forbids $S$ on vertices $u_1,\ldots,u_r$ with respect to $(\cL(u_1)\times\ldots\times\cL(u_r),\cL(v))$.
Add $p_{v,S}=0$ to $P$.
This completes the costruction of $P$.
By \cref{obs:minimal-neighbors} and the definition of $P$, there is one-to-one correspondence between the choice assignments of $\vect{y}$ satisfying all equalities if $P$ and list homomorphims from $(G[X],\cL)\to H$ that can be extended to $(G,\cL)$.

Now we apply \cref{thm:d-poly-root-csp} to $P$, and in polynomial time we obtain $P' \subseteq P$ with the same set of satisfying assignments and which has at most $\Oh((h\cdot k)^d)$ polynomials.

\paragraph{Construction of $G'$.} Based on $P'$, we now show how to construct our new instance $(G',\cL)$.
We start with $G[X]$.
Furthermore, for every $v \in I$, add $v$ to $G'$ if one of the polynomials $p_{v,S}$ is contained in $P'$. 
If this is the case, we also add to $G'$ all edges connecting $v$ to all vertices $u_1,\ldots,u_r$ such that some $p_{v,S}\in P'$ was introduced for $u_1,\ldots,u_r$.
This completes the definition of $G'$.
Let us show that the instances $(G,\cL)$ and $(G',\cL)$ of \listcoloring{H} are equivalent.

\begin{claim}
$(G,\cL)\to H$ if and only if $(G',\cL)\to H$.
\end{claim}

\begin{claimproof}
Since $G'$ is a subgraph of $G$, clearly if $(G,\cL)\to H$, then also $(G',\cL)\to H$.
Thus it suffices to show that if $(G',\cL)\to H$, then $(G,\cL)\to H$.
So suppose there is a list homomorphism $\vphi':(G',\cL)\to H$.
We start with defining $\vphi:(G,\cL)\to H$, by setting $\vphi|_X=\vphi'$.
We will show that $\vphi'$ can be extended to a list homomorphism $\vphi:(G,\cL)\to H$.
Since $I=V(G)\setminus X$ is an independent set, it suffices to show, that for every vertex $v\in I$, we can extend $\vphi$ to $v$.

First, consider the choice assignment of $\vect{y}$ corresponding to $\vphi$, i.e., $y_{v_i,x_j}=1$ if and only if $\vphi(v_i)=x_j$.
Observe that it satisfies all the equalities of $P'$.
Indeed, since $\vphi$ respects the lists, the choice assignment of $\vect{y}$ satisfies all list polynomials.
So now consider a polynomial $p_{v,S}$ introduced for vertex $v\in I'$ and its neigbors $u_1,\ldots,u_r$, so that $p_{v,S}$ forbids some tuple $S=(s_1,\ldots,s_r)$ on vertices $u_1,\ldots,u_r$.
Let $S'=(s'_1,\ldots,s'_r)=(\vphi(u_1),\ldots, \vphi(u_r))$.
Since $\vphi'$ is a homomorphism on $G'$, then $\bigcap_{i=1}^r N_H(s'_i)\cap \cL(v) \neq \emptyset$.
Therefore, by the definition of $p_{S,v}$, it holds that $p_{v,S}(\vect{y})=0$ as desired.
Since the choice assignment of $\vect{y}$ satisfies all equalities of $P'$, then it also satisfies all the equalities of $P$.

Now suppose that $\vphi$ cannot be extended to some vertex of $I$, i.e., there is $v\in I$ such that $\bigcup_{u\in N_G(v)} N_H(\vphi(u))\cap \cL(v)=\emptyset$.
Furthermore, let $U$ be a minimal subset of $N_G(v)$ such that $\bigcup_{u\in U} N_H(\vphi(u))\cap \cL(v)=\emptyset$, and let $U=\{u_1,\ldots,u_r\}$.
Note that $r\leq c^*(H)$.
Then there is equality $p_{v,S}=0$ in $P$, where $S=(\vphi(u_1),\ldots,\vphi(u_r))$ and $p_{S,v}$ forbids $S$ on $(u_1,\ldots,u_r)$.
Since the choice assignment of $\vect{y}$ corresponds to $\vphi$, by the definition of $p_{v,S}$, it holds that $p_{v,S}=1$, a contradiction.
This completes the proof.
\end{claimproof}

It remains to bound the size of $G'$.
The subgraph $G'[X]$ has at most $\Oh(k^2)$ vertices and edges, and for each polynomial in $P'$ we add at most $c^*(H)$ edges and at most one vertex.
Hence, the total number of vertices and edges is bounded by $\Oh(k^2 + |P'| \cdot c^*(H)) = \Oh((h \cdot k)^d \cdot h) =\Oh(k^d)$, which completes the proof.
\end{proof}

Using \cref{thm:kernelization-by-polynomials} we can reprove \cref{thm:marking}.
 
\begin{proof}[Alternative proof of \cref{thm:marking}]
Let $L\subseteq V(H)$ and let $S=(s_1,\ldots,s_r)$ be a sequence of vertices of $H$ with $r\leq c^*(H)$ and such that $S$ does not have a common neighbor in $L$.
We define polynomial $p_S$ that forbidds coloring vertices $v_1,\ldots,v_r$ with $s_1,\ldots,s_r$ respectively, by setting $p_S(y):= \prod_{i=1}^{r} y_{v_i,s_i}$. Clearly the degree of $p_S$ is $r\leq c^*(H)$. Moreover, $p_S$ is equal to $1$ if and only $y_{v_i,s_i}=1$ for every $i\in [r]$, which corresponds to coloring vertices $v_1,\ldots,v_r$ respectively with $s_1,\ldots,s_r$. Thus every sequence $S$ of length at most $c^*$ can be forbidden by a polynomial of degree at most $c^*$ and the statement follows from \cref{thm:kernelization-by-polynomials}.
\end{proof}

\subsection{Specific graph classes}
In this section we consider some specific graph classes and we show that by the polynomial method we can obtain kernels of size $\Oh(k^{d^*(H)})$.

As building blocks of the constructed polynomials, we will use the following result of Jansen and Pieterse~\cite{JansenP19Coloring} -- we state it in the form that is suitable for us.

\begin{lemma}[{{\cite[Lemma 5]{JansenP19Coloring}}}]\label{lemma:poly-local}
Let $G,H$ be graphs, let $S \subseteq V(H)$, and let $(v_1,\ldots,v_r) \in V(G)^r$, where $r=|S|+1$.
There is a polynomial $p_S(\vect{y})$ over GF(2) of degree $|S|$, where $\vect{y}$ is the vector containing variables $y_{v,u}$, for $v \in V(G)$, $u \in V(H)$, such that for any partial choice assignment of $\vect{y}$, we have $p_S(\vect{y})=1$ if and only if every color $s \in S$ is used precisely once in $\vect{y}$ among vertices $v_1, \ldots, v_r$.
\end{lemma}
\subsubsection{Powers of cycles}
\label{sec:powers-of-cycles}
As the first natural class of graphs, we consider cycles, and then also powers of cycles.

For an integer $k$, define $\llbracket k \rrbracket := \{0,1,\ldots,k-1\}$.
We will perform all arithmetic operations on elements of $\llbracket k \rrbracket$ modulo $k$.
For $u,v \in \llbracket k \rrbracket$, we define $\dist(u,v) := \min \{ |u-v|, k-|u-v| \}$. Intuitively, $\dist(u,v)$ denotes the cyclic distance between $u$ and $v$.

For integers $k \geq 3$ and $p \geq 1$, let $C^p_k$ denote the $p$-th power of the $k$-cycle, i.e., the graph with vertex set $\llbracket k \rrbracket$ and the edge set $\{ uv ~|~ \dist(u,v) \leq p\}$.
We will refer to the Hamiltonian cycle of $C^p_k$ with consecutive vertices $0,1,\ldots,k-1$ as the \emph{frame} of $C^p_k$.

As a warm-up, we first consider the case $p=1$, i.e., the case of cycles.
Let us point out that the case of \listcoloring{C_3} is equivalent to \listcoloring{3}, which was studied by Jansen and Pieterse~\cite{JansenP19Coloring}.
Furthermore, the graph $C_4$ is a bi-arc graph, and thus \listcoloring{C_4} can be solved in polynomial time~\cite{DBLP:journals/jgt/FederHH03}.
Therefore, we focus on $C_k$ with $\ell \geq 5$ -- all such graphs are non-bi-arc~\cite{DBLP:journals/jgt/FederHH03}.
Recall that, for $k\geq 5$, by \cref{obs:cycles-c-d}, it only suffices to consider the case $k=6$, as for all other values of $k$, we have that $d^*(C_k)=c^*(C_k)$.
Moreover, recall that by \cref{obs:cycles-c-d}, we have $c^*(C_6)=d^*(H)+1=3$.

\begin{lemma}\label{lem:cycles-poly}
Let $r\leq c^*(C_6)=3$, let $L_1,\ldots,L_r\subseteq V(C_6)$ be incomparable sets, and let $S\in L_1\times\ldots \times L_r$ be an $r$-tuple without a common neighbor in some set $L$.
Then $S$ can be forbidden by a polynomial of degree $2$ with respect to $(L_1\times L_2\times L_3,L)$.
\end{lemma}

\begin{proof}
Let us point out that we can assume that $S$ is minimal without a common neigbor in $L$, otherwise we can forbid a subsequence of $S$.
Furthermore, if $r\leq 2$, then in order to forbid $S=(s_1,s_2)$ on $(v_1,v_2)$, we can simply set $p(y)=y_{v_1,s_1}\cdot y_{v_2,s_2}$, which is clearly a polynomial of degree $2$ that forbidds $S$ on $(v_1,v_2)$.

So now we will show that any $S\in L_1\times L_2 \times L_3$, which is minimal without a common neighbor in some $L$ can be forbidden by a polynomial of degree $2$ with respect to $(L_1\times L_2 \times L_3,L)$.
Recall that, as observed in the proof of \cref{obs:cycles-c-d}, such a sequence $S$ consists of three distinct vertices, either $0,2,4$ or $1,3,5$.
By symmetry, it suffices to consider the case of $0,2,4$.
Furthermore, we will show how to forbid the set $\{0,2,4\}$ (instead of forbidding a sequence) on vertices $v_1,v_2,v_3$. 

We call \cref{lemma:poly-local} to construct polynomials $p_{\{i,j\}}$ of degree $2$, for every distinct $i,j\in \{0,2,4\}$, such that for any choice assignment $\mathbf{y}$, it holds that $p_{\{i,j\}}(\mathbf{y})\equiv_21$ if and only if $\mathbf{y}$ corresponds to a coloring of $v_1,v_2,v_3$ such that each of $i,j$ occurs on $v_1,v_2,v_3$ exactly once.
We set $p(\mathbf{y})=p_{\{0,2\}}(\mathbf{y})+p_{\{2,4\}}(\mathbf{y})+p_{\{0,4\}}(\mathbf{y})$.
Clearly, $p$ is of degree $2$,
Furthermore, for a choice assignment $\mathbf{y}$ corresponding to coloring of $v_1,v_2,v_3$ such that each of the colors $0,2,4$ appears precisely once, each polynomial in the sum evaluates to $1$, and $p(\mathbf{y})\equiv_21$. 
Moreover, if $S'$ is a set which has a common neighbor, then note that it must hold $|S'|\leq 2$, and thus for a choice assignment $\mathbf{y}$ that corresponds to $S'$, none of the polynomials in the sum evaluates to $1$, so $p(\mathbf{y})\equiv_20$.
Therefore, $p$ satisfies the desired properties.
This completes the proof.
\end{proof}

Therefore, combining \cref{thm:lower-bound}, \cref{thm:kernelization-by-polynomials}, \cref{obs:cycles-c-d}, and \cref{lem:cycles-poly}, we obtain the following.
 
\begin{theorem}\label{thm:cycles}
Let $\ell\geq 5$.
Then \listcoloring{C_{\ell}} parameterized by the size $k$ of the minimum vertex cover of the input graph admits a kernel with $\Oh(k^{d^*(C_\ell)})$ vertices and edges, but does not admit a kernel with $\Oh(k^{d^*(C_\ell)-\eps})$ vertices and edges, for any $\eps>0$, unless \containment.
\end{theorem}

Now we proceed to larger values of $p$.
Let us first analyze, how the parameter $c^*$ behaves for such graphs.

\begin{lemma}\label{lem:forbidden-sets}
Let $k \geq 7$, $p \geq 2$ be integers such that $6p < k$. Then $\cc(C_k^p) = p+1$.

\noindent
Furthermore, if $L \subseteq \llbracket k \rrbracket$ and $S \subseteq \llbracket k \rrbracket$ are such that $|S|=p+1$ and $S$ is a minimal set with no common neighbor in $L$,
then either
\begin{enumerate}[(a)]
\item $S = \{i,i+1,\ldots,i+p\}$ for some $i \in \llbracket k \rrbracket$, or \label{case:consecutive}
\item $S = \{i,i+2,\ldots,i+p,i+p+2\}$ for some $i \in \llbracket k \rrbracket$. \label{case:nonconsecutive}
\end{enumerate}
\end{lemma}
\begin{proof}
Let $C^p_k$, $S$, and $L$ be as in the statement of the lemma.

First, observe that $\cc(C^p_k) \geq p+1$. Indeed, it is enough consider $L=\llbracket k \rrbracket$ and the sets $S_1 = \{0,1,\ldots,p\}$ 
or $S_2 = \{0,2,\ldots,p,p+2\}$.
Since $k > 6p$, we observe that $S_1$ has no common neighbor, but for every $u \in S_1$, the vertex $u$ is a common neighbor of $S_1 \setminus \{u\}$.
Similarly, $S_2$ has no common neighbor, but for every $u \in S_2$, the set $S_2 \setminus \{u\}$ has a common neighbor $w$.
Indeed, if $u = 0$, then $w = p+1$, if $u = p+1$, then $w=1$, and in all other cases $w = u$.
Note that these two examples correspond to cases~\eqref{case:consecutive} and~\eqref{case:nonconsecutive} in the second part of the statement of the lemma.

Now let us show that $\cc(C^p_k) \leq p+1$.
Let $L \subseteq \llbracket k \rrbracket$ and $S \subseteq \llbracket k \rrbracket$ are such that $S$ is a minimal set with no common neighbor in $L$,
and $S$ is maximum possible. By the previous paragraph we can assume that $|S| \geq p+1 \geq 3$.

First, consider the case that $S$ contains $p+1$ consecutive vertices of the frame, by symmetry, we can assume that $\{0,1,\ldots,p\} \subseteq S$.
Suppose that there is some $u \in S \setminus \{0,1,\ldots,p\}$.
However, the subset $\{0,1,\ldots,p\}$ of $S$ has no common neighbor in $\llbracket k \rrbracket$ and thus in $L$, which contradicts the minimality of $S$.
Therefore, $S = \{0,1,\ldots,p\}$ and we obtain the case~\eqref{case:consecutive} in the second statement of the lemma.

So since now we can assume that $S$ does not contain $p+1$ consecutive vertices of the frame.
Let $u,v$ be the vertices of $S$, for which $\dist(u,v)$ is maximized. By symmetry of $C^p_k$ we can assume that $u=0$ and $0 < v \leq (k-1)/2$, as $|S| \geq 2$.

First suppose that $\dist(0,v) = v > 2p$.
Then we observe that $0$ and $v$ have no common neighbors in $\llbracket k \rrbracket$ (and thus in $L$),
so, by the minimality of $S$,
we obtain that $S = \{0,v\}$ and so $|S| = 2 <3$, a contradiction.
So from now on assume that $\dist(0,v) = v \leq 2p$.

Since $k \geq 6p$,  we claim that for all $w \in S$, it holds that $\dist(0,w) + \dist(w,v) = \dist(0,v) = v$.
In other words, $w$ lies on the shortest $0$-$v$-path in the frame of $C^p_k$, i.e., $0 \leq w \leq v$.
Indeed, if this is not the case, then either $\dist(0,w) > \dist(0,v)$ or $\dist(v,w) > \dist(0,v)$, which contradicts our choice of $u=0$ and $v$.

Now let us observe that that $S \subseteq \{0,v\} \cup (N(0) \cap N(v))$.
Indeed, suppose that there is $z\in S$ such that $z \notin \{0,v\} \cup (N(0) \cap N(v))$ and consider the set $S' := S\setminus\{z\}$.
By the minimality of $S$, the set $S'$ has a common neighbor $w\in L$.
Note that $w\neq z$ as $z\notin N(0)\cap N(v)$ and $w$ must be in particular a common neighbor of $0$ and $v$.
Furthermore, as for every common neighbor of $0$ and $v$, it must hold that $\dist(0,w)\leq p$ and $\dist(w,v)\leq p$.
Since we also have $0<z<v$ as $z\in S$, then $\dist(z,w)\leq p$, so $z \in N(w)$.
Therefore, $w$ is also a common neighbor of $S$, a contradiction.

Let us estimate the size of the set $|S|$.
Recall that $S$ contains only the vertices $w$, such that $0 \leq w \leq v$, and are common neighbors of $0$ and $v$ or are equal to $0$ or $v$.
Observe that $0$ has precisely $p$ neighbors $w$, such that $w > 0$,
and similarly $v$ has precisely $p$ neighbors $w$, such that $w < v$.
Thus $|S| \leq p+2$.

If $|S| = p+2$, then $S$ consists of $0,v$ and precisely $p$ vertices $w$ that have to satisfy $0<w\leq p$ and $v-p\leq w<v$.
This is only possible for $v = p+1$ and $S = \{0,1,\ldots,p+1\}$.
But then $S$ contains $p+1$ consecutive vertices of the frame, a contradiction.

So now assume that $|S| = p+1$, i.e., $S$ contains $0,v$ and $p-1$ vertices $w$ that satisfy $0<w\leq p$ and $v-p\leq w<v$.
As there have to be $p-1$ vertices $w$ satisfying $v-p\leq w\leq p$, it must hold that $v\leq p+2$.
On the other hand, since we consider the case that $S$ does not contain $p+1$ consecutive vertices of the frame, we observe that $v\geq p+1$, and thus $v\in \{p+1,p+2\}$.

Suppose first that $v=p+1$.
Since $|S|=p+1$, there is precisely one $0<u<p+1$ such that $u\notin S$.
Note that $u\notin L$, as otherwise $S$, which contains only vertices $w$ such that $0\leq w\leq p+1$, has a common neighbor $u\in L$, a contradiction.
Now consider the set $S'=S\setminus \{0\}$.
By the minimality of $S$, the set $S'$ has a common neighbor $w\in L$.
Note that $w\neq u$ and $w\notin S'$, and thus either $w<1$ or $w>v=p+1$.
We cannot have $w<1$, as then $w$ is non-adjacent to $v$, so we have $w>p+1$.
Then $1\notin S'$, as $1$ cannot be adjacent to $w$, and this is only possible if $u=1$.
Now let $S''=S\setminus \{p+1\}$.
Again, by the minimality of $S$, there is a common neighbor $z\in L$ of $S''$.
Since $z\notin S''$ and $z\neq u$, we either have $z<0$ or $z>p$.
We cannot have $z>p$, as then $z$ is non-adjacent to $0\in S''$, and thus $z<0$.
But then $z$ is non-adjacent to $p\geq 2\neq u$, which is a contradiction with the fact that $z\in S''$ and $z$ being the common neighbor of $S''$.

So it remains to consider the case $v=p+2$.
As every vertex $w\in S\setminus \{0,v\}$ has to be adjacent to both $0$ and $v$, all such vertices have to satisfy $v-p\leq w\leq p$, which for $v=p+2$ is $2\leq w \leq p$.
Since $|S|=p+1$, the only possiblity is that $S=\{0,2,3,\ldots,p,p+2\}$.
This gives case~\eqref{case:nonconsecutive} in the second statement of the lemma, which completes the proof.
\end{proof}

 
Now we will show how to forbid the sets of types~\eqref{case:consecutive} and~\eqref{case:nonconsecutive} by low-degree polynomials.

\begin{lemma}\label{lem:powers-poly}
Let $p \geq 2$, $k > 6p$, $L_i \subseteq V(C_k^p)$ be an incomparable set for $i\in [p+1]$, and let $S\in F=L_1 \times \ldots \times L_{p+1}$ be minimal with no common neighbor in $L\subseteq V(C_k^p)$. Then $S$ can be forbidden by a polynomial of degree at most $p$, with respect to $(F,L)$.
\end{lemma}

\begin{proof}
Let $S = (s_1,\ldots,s_{p+1})$. By minimality of $S$, it holds that the set $\{s_1,\ldots,s_{p+1}\}$ has $p+1$ distinct elements. We will construct a polynomial with a stronger property than required, i.e., in assignment corresponding to any coloring of vertices $v_1,\ldots,v_{p+1}$, if the set $\{s_1,\ldots,s_{p+1}\}$ occurs, then the polynomial equals to $1$ (modulo $2$), and if a set with a common neighbor in $C^p_k$ occurs, then the polynomial equals to $0$ (modulo $2$). Since we focus on sets rather than sequences, we will treat $S$ as the set $\{s_1,\ldots, s_{p+1}\}$.  

Recall that we use variables $y_{v_i,c}$ for $i \in [p+1], c \in \llbracket k \rrbracket$ representing that vertex $i$ has color $c$.

For $i,j\in \llbracket k \rrbracket$, we set $A_{i,j}:=\{i,j,j+1,\ldots,j+p-2\}$.
Let us also define $I\subseteq \llbracket k \rrbracket^2$ to be the set of of these pairs $(i,j)$ such that: (i)  $2 \leq \dist(i,j) \leq 2p$, and (ii) $\dist(i,j) < \dist(i,j+1)$.
For every $(i,j)\in I$, we call \cref{lemma:poly-local} to construct the polynomial $p_{A_{i,j}}$, which evaluates to $1$ (modulo $2$) if each of the colors in $A_{i,j}$ is used exactly once by choice assignment $\mathbf{y}$ (implying also that one color not from this set is used), and evaluates to $0$ otherwise.
 
 
We construct the polynomial $p$ as follows:
\begin{align*}
f(\mathbf{y}) = \sum_{(i,j) \in I} p_{A_{i,j}}(\mathbf{y}).
\end{align*}
 
We will prove that for a choice assignment $\mathbf{y}$ corresponding to a coloring of $v_1,\ldots,v_{p+1}$ such that the set $S$ occurs, we have $f(\mathbf{y})\equiv_21$, and for assignment corresponding to coloring such that a set $S'$ with a common neighbor occurs, we have $f(\mathbf{y})\equiv_20$. 

\paragraph{Case 1: $p=2$.} By \cref{lem:forbidden-sets}, the set $S$ is of one of the following types: $\{i,i+1,i+2\}$ or $\{i,i+2,i+4\}$.
Furthermore, we have that $I=\{(i,i+2), (i,i+3), (i,i+4) \ | \ i\in \llbracket k \rrbracket\}$, and for $i\in \llbracket k \rrbracket$, we have $A_{i,i+2}=\{i,i+2\}$, $A_{i,i+3}=\{i,i+3\}$, and $A_{i,i+4}=\{i,i+4\}$.
So for an assignment $\mathbf{y}$ corresponding to $S$ of type $\{i,i+1,i+2\}$, the only polynomial in the sum evaluating to $1$ is $p_{A_{i,i+2}}$, and for $S$ of type $\{i,i+2,i+4\}$, exactly three polynomials evaluate to $1$: $p_{A_{i,i+2}}$, $p_{A_{i,i+4}}$, and $p_{A_{i+2,i+4}}$.
In both cases we obtain that $f(\mathbf{y})\equiv_21$.

It remains to show that if a set $S'$ with a common neighbor occurs, then $f(\mathbf{y})\equiv_20$. If $|S'|\leq 2$, then no polynomial in the sum evaluates to $1$, thus $f(\mathbf{y})\equiv_20$. And if $|S'|=3$, then it can be verified that $S'$ is of one of the following types: (i) $\{i,i+1,i+4\}$, (ii) $\{i,i+3,i+4\}$, (iii) $\{i,i+1,i+3\}$, and (iv) $\{i,i+2,i+3\}$ -- note that since $S'$ has a common neighbor, for any vertices $uv\in S'$ it must hold $\dist(u,v)\leq4$. In every case there are precisely two polynomials in the sum evaluating to $1$, respectively: (i) $p_{A_{i,i+4}}$ and $p_{A_{i+1,i+4}}$, (ii) $p_{A_{i,i+3}}$ and $p_{A_{i,i+4}}$, (iii) $p_{A_{i,i+3}}$ and $p_{A_{i+1,i+3}}$, (iv) $p_{A_{i,i+2}}$ and $p_{A_{i,i+3}}$. Therefore, in each case we obtain $f(\mathbf{y})\equiv_20$.
  
\paragraph{Case 2: $p\geq 3$.} So since now assume that $p\geq 3$ and let us prove that for choice assignment $\mathbf{y}$ corresponding to a coloring of $v_1,\ldots,v_{p+1}$ such that the set $S$ occurs, we have $f(\mathbf{y})\equiv_21$.
As $S$ is a minimal set without a common neighbor in $L$ of size $p+1$, then by \cref{lem:forbidden-sets}, the set $S$ is of one of the following types: $\{i,i+1,\ldots,i+p\}$ or $\{i,i+2,\ldots,i+p,i+p+2\}$.
In both cases, we have $p_{A_{i,i+2}}(\mathbf{y})\equiv_21$, (recall that $A_{i,i+2}=\{i,i+2,\ldots,i+p\}$) and it can be verified, that this is the only polynomial in the sum that evaluates to $1$, so $f(\mathbf{y})\equiv_21$.
 
It remains to show that if a set $S'$ has a common neighbor in $C_k^p$, then it is not forbidden by $f$.
Let $\mathbf{y}$ be a choice assignment using the colors in $S'$.
If $|S'| \leq p$, observe that all terms of $p$ are always zero as some color occurs twice on a set of $p + 1$ vertices by the pigeon hole principle, and thus $f(\mathbf{y})\equiv_20$.
Furthermore, if $S'$ does not contain $p-1$ consecutive vertices, then every polynomial in the sum equals $0$, and again $f(\mathbf{y})\equiv_20$ as required.

So since now we assume that $|S'| = p+1$ and $S'$ contains $p-1$ consecutive vertices.
By symmetry, we can assume that $S'$ contains $p+1,\ldots,2p-1$.
Since $|S'|=p+1$, then there are two more vertices $u,v$ in $S'$. 

Assume first that there are $p$ consecutive vertices in $S'$, again by symmetry we can assume that $u=2p$, i.e., $S'$ contains the vertices $p+1,\ldots,2p$.
Note that the only common neighbors of $\{p+1,\ldots,2p\}$ are $p$ and $2p+1$.
Therefore $v$ must be adjacent to $p$ or $2p+1$, so either $0\leq v<p$ or $2p+1<v\leq 3p+1$.
If the first case applies, then the only polynomials evaluating to $1$ are: $p_{A_{v,p+1}}$ and $p_{A_{v,p+2}}$, i.e., the polynomials for the sets $\{v,p+1,p+2,\ldots,2p-1\}$ and $\{v,p+2,p+3,\ldots,2p\}$, so $f(\mathbf{y})\equiv_20$ as required.
If the other case applies, then no polynomial evaluates to $1$, so again $f(\mathbf{y})\equiv_20$ as required.

Finally, assume that there are no $p$ consecutive vertices in $S'$, so $u,v \notin \{p,2p\}$.
Now observe that the only common neighbors of $\{p+1,\ldots,2p-1\}$ are $p-1$, $p$, $2p$, and $2p+1$.
Therefore $u,v$ must be both adjacent either to one of $\{2p,2p+1\}$ or to one of $\{p-1,p\}$.
If the first case applies, then we have that $u,v>2p$.
If $p>3$, then this means that no polynomial in the sum evaluates to $1$, so $f(\mathbf{y})\equiv_20$ as required.
If $p=3$, the only case when there is some polynomial that does not evaluate to $0$ is when the vertices $u,v$ are consecutive, say $u<v$.
In such a case there are exactly two polynomials evaluating to $1$: the ones introduced for the sets $A_{p+1,u}=\{p+1,u,v\}$ and $A_{p+2,u}=\{p+2,u,v\}$, so again $f(\mathbf{y})\equiv_20$.
In the other case we have $\dist(u,p-1)<\dist(u,p)\leq p+1$ and $\dist(v,p-1)<\dist(v,p)\leq p+1$, and there are exactly two polynomials that evaluate to $1$: the ones introduced for the sets $A_{u,p+1}=\{u,p+1,\ldots,2p-1\}$ and $A_{v,p+1}=\{v,p+1,\ldots,2p-1\}$, so again $f(\mathbf{y})\equiv_20$ as required.
This completes the proof.
\end{proof}
 
Let us also mention that for $p \geq 2$, $\ell > 6p$, the graph $C_{\ell}^p$ is non-bi-arc.
Therefore, combining \cref{thm:lower-bound}, \cref{thm:kernelization-by-polynomials}, \cref{lem:forbidden-sets}, and \cref{lem:powers-poly}, we obtain the following.

\begin{theorem}\label{thm:powers-cycles}
Let $p \geq 2$, $\ell > 6p$.
Then \listcoloring{C_{\ell}^p} parameterized by the size $k$ of the minimum vertex cover of the input graph admits a kernel with $\Oh(k^{p})$ vertices and edges, but does not admit a kernel with $\Oh(k^{p-\eps})$ vertices and edges, for any $\eps>0$, unless \containment.
\end{theorem}

\subsubsection{Small-degree graphs $H$}
\label{sec:bounded-degree}
Recall that if $d^*(H)=c^*(H)$, tight bounds are obtained by the combination of \cref{thm:marking} and \cref{thm:lower-bound}.
Thus from now on assume that $d^*(H)=c^*(H)-1$.

We now consider graphs where $c^*(H)$ is close to or equal to the
maximum degree $\Delta(H)$.
By \cref{lem:delta-c-plus1}, we have $\cc(H) \leq \Delta(H)+1$.
If $\cc(H) = \Delta(H)+1$, then $\dd(H) = \Delta(H)$, and tight bounds follow from combining \cref{thm:kernel-delta} due to Jansen and Pieterse~\cite{JansenP19Coloring} and \cref{thm:lower-bound}.

So from now on we focus on the case $\cc(H)=\Delta(H)$  and show that such graphs have a kernel with $\Oh(k^{\dd(H)})$ vertices and edges.
By \cref{thm:kernelization-by-polynomials},
it suffices to show the existence of appropriate polynomials of degree $d^*(H)$.

We begin by a useful structural property. 

\begin{lemma} \label{lemma:max-deg-property}
Let $H$ be a graph such that $d^*(H)+1=c^*(H)=\Delta(H)$.  
Let $S=N(v)$ for some $v \in V(H)$ with $\deg(v)=\Delta(H)$.
Then there exists some $u \in S$ such that $S\setminus\{u\} \subseteq N(v')$
for $v' \in V(H)$ implies $N(v') \subseteq N(v)$. 
\end{lemma}
\begin{proof}
Assume to the contrary that for every $u \in S$ there exists a vertex $v_u \neq v \in V(H)$ such that $\deg(v_u)=\deg(v)$ and $N(v_u) = (S\setminus\{u\}) \cup u'$ for some $u' \in V(H)$, $u' \notin S$.
We split into two cases.
First, assume that $\{v_u \mid u \in S\}$ has a common neighbor $z$.
Note that $z\notin S$, as then for $u=v$, we have that $N(v_u)=S\setminus\{u\}\cup\{u'\}=S\setminus\{z\}\cup\{u'\}$, and thus $v_u$ cannot be adjacent to $z$.
Moreover, for every $u \in S$, since $\deg(v_u) \geq |S\setminus\{u\}\cup\{u'\}|=\Delta(H)$, it holds that $u'=z$.
Now consider the set $S'=S\cup\{z\}$.
It has no common neighbors since $|S'| > \Delta(H)$, but $S'\setminus\{z\}=S=N(v)$ has a common neighbor $v$, and for every $u \in S$, the set $S'\setminus\{u\}=S\setminus\{u\}\cup\{z\}$ has a common neighbor $v_u$.
Then $S'$ contradicts the claim that $c^*(H)=\Delta(H)$. 

Otherwise, the set $X=\{v_u \mid u \in S\}$ has no common neighbor.
Note that $|X|=|S|$, as all vertices $v_u$ must be distinct by the degree bound on $H$.
On the other hand, let $v_u'=v$ for every $v_u \in X$.
Then $X\setminus\{v_u\}\cup\{v_u'\}=X\setminus\{v_u\}\cup\{v\}$ has a common neighbor $u$, and this obviously remains true if we replace more vertices $v_i \in X$ by $v_i'=v$ (since the resulting set is then a strict subset of $X\setminus\{v_u\}\cup\{v\}$).
Finally, $v_u$ and $v$ are incomparable for every $u \in S$, since $u \in N(v) \setminus N(v_u)$ and $u' \in N(v_u) \setminus N(v)$.
Thus $X$ is a lower bound structure of size $|X|=\Delta(H)=c^*(H)>d^*(H)$ in $H$, a contradiction.
\end{proof}

Rephrased, let $\mathcal{S}=\{S \subseteq V(H): |S|=c^*(H), S \text{ has a common neighbor}\}$. 
Then for every $S \in \mathcal{S}$ there exists a set $S' \in \delta S$
such that $S' \subseteq S''$ for $S'' \in \mathcal{S}$ only for $S''=S$.
We will see that this guarantees the existence of an appropriate polynomial of degree $d^*(H)$.
We have the following.

\begin{lemma} \label{lem:bd-poly}
Let $H$ be a graph with $d^*(H)+1=c^*(H)=\Delta(H)$.
Let $L \subseteq V(H)$, $F = L_1 \times \ldots \times L_r$, and a sequence $(v_1, \ldots, v_r) \in X^r$ be given, where $L_i \subseteq V(H)$ is an incomparable set for each $i \in [r]$.
Let $S_0 \in F$ be a sequence which has no common neighbor in $L$.
Then $S_0$ can be forbidden on $(v_1,\ldots,v_r)$ by a polynomial of degree $d^*(H)$ with respect to $(F,L)$.
\end{lemma}
\begin{proof}
Let $S_0=(s_1,\ldots,s_r)$.
To simplify notation in the proof, we will occasionally treat $S_0$ and elements $S \in F$
as sets rather than sequences, as the construction is largely independent of element order.
Note that if $|S_0|<r$, then we can forbid a subsequence of $S_0$.
Also, if $r<c^*(H)$, then we can forbid $S_0$ trivially by the polynomial $p(\vect{y})=\prod_{i=1}^r y_{v_i,s_i}$.
Hence assume $|S_0|=r=c^*(H)=d^*(H)+1$.

For a set $S\in \binom{V(H)}{r}$, we define $\overrightarrow{S}$ to be the 0-1 vector indexed by elements of $\binom{V(H)}{d^*(H)}$, such that for $S'\in \binom{V(H)}{d^*(H)}$, we have $\oS[S']=1$ if and only if $S'\in \delta S$; in other words, $\oS$ can be seen as the indicator vector of $\delta S$.
We will use the following two facts.
\begin{enumerate}[(1.)]
\item For two distinct sets $S_1,S_2\in \binom{V(H)}{r}$, there is at most one $S'\in \binom{V(H)}{d^*(H)}$ such that $\overrightarrow{S_1}[S']=\overrightarrow{S_2}[S']=1$, as $|\delta S_1\cap \delta S_2|\leq 1$, otherwise the sets cannot be distinct.
\item By \cref{lemma:max-deg-property}, for every $S_1\in \mathcal{S}$, there is $S'\in \binom{V(H)}{d^*(H)}$ such that $\overrightarrow{S_1}[S']=1$, and for every other $S_2 \in \mathcal{S}$, we have $\overrightarrow{S_2}[S']=0$.
\end{enumerate}

\paragraph{Defining the polynomial.} Let $\mathcal{S} = \{S \in F \mid |S|=r, S \text{ has a common neighbor in } L\}$.
We will construct a polynomial over GF(2) as a sum of polynomials $p_S(\vect{y})$ from \cref{lemma:poly-local} as
\[
    p(\vect{y}) = \sum_{S \in {\binom{V(H)}{d^*(H)}}} p_S(\vect{y})\cdot \mathcal{X}[S],
\]
where $\mathcal{X}$ is a 0-1 vector indexed by the elements of $\binom{V(H)}{d^*(H)}$.
We then want $p$ to satisfy the following: (i) $p(\vect{y})=1$ if $\vect{y}$ is a choice assignment representing $S_0$,
and (ii) $p(\vect{y})=0$ if $\vect{y}$ is a choice assignment representing some $S \in \mathcal{S}$ (or if $\vect{y}$ uses fewer than $r$ distinct colors on $v_1, \ldots, v_r$).

Such a vector ${\mathcal X}$  exists if there is a solution to the system of equations
\[
 \forall S \in \mathcal{S} \quad \sum_{S' \in \delta S} \mathcal{X}[S'] \equiv_2 0,
\]
and
\[
\sum_{S' \in \delta S_0} \mathcal{X}[S'] \equiv_2 1.
\]
Indeed, if this system has a solution, then $p(\vect{y})\equiv_2 1$ only if $\vect{y}$ uses on $v_1,\ldots,v_r$ a set of $r$ distinct values, which is not included in $\mathcal{S}$, and in particular $p(\vect{y})\equiv_2 1$ for the values in $S_0$.

\paragraph{Existence of a solution to the system of equations.} So it remains to prove that the system has a solution.
Let us rewrite it in the following form.
\begin{equation*}
\begin{cases}
    \oS\cdot \mathcal{X} \equiv_2 0, & \forall S \in \mathcal{S},\\
    \overrightarrow{S_0}\cdot \mathcal{X} \equiv_2 1,
  \end{cases}
\end{equation*}
where $\cdot$ denotes the scalar product.

For contradiction, assume that the system has no solution.
Then there is a non-empty subset $\mathcal{S'}\subseteq \mathcal{S}\cup \{S_0\}$ such that the corresponding vectors form a linearly dependent set, i.e., $\sum_{S\in \mathcal{S'}} \oS \equiv_2 0$ (here $0$ states for the zero-vector).

Observe that:
\begin{description}
\item[($\star$)] For every $S_1\in \mathcal{S'}$, and for every $S'\in \delta S_1$, there is $S_2\in \mathcal{S'}\setminus \{S_1\}$ such that $S' \in \delta S_2$.
\end{description}
Indeed, it must hold that $\sum_{S\in \mathcal{S'}} \oS[S'] \equiv_2 0$, so if $\overrightarrow{S_1}[S']=1$, there must be another $S_2\in \mathcal{S'}$ with $\overrightarrow{S_2}[S']=1$.
In particular, by (2.), $\mathcal{S'}$ must include $S_0$; for every $S\in \mathcal{S}\cap\mathcal{S'}$, for a private subset $S'\in \delta S$, it must hold that $S'\in \delta S_0$.

First assume that there is a set $S_0' \in \delta S_0$ that is \emph{not} the private subset of any $S \in \mathcal{S'}$.
Then, by ($\star$), there is $S \in \mathcal{S}$ such that $S_0' \in \delta S$.
By (1.), for any such set $S$, its private subset $S'$ does not occur in $S_0$, as we have $|\delta S_0\cap \delta S|\leq 1$.
Moreover, since $S'$ is a private subset of $S$, there is no other set in $\mathcal{S}$ (and thus in $\mathcal{S'}$) that contains $S'$.
Then we cannot have $S\in \mathcal{S'}$, otherwise we have $\sum_{S\in \mathcal{S'}} \oS[S']\equiv_2 1$, a contradiction.

So we can now assume that for every $S_0' \in \delta S_0$ there is precisely one set $S \in \mathcal{S'}\setminus\{S_0\}$ such that $S_0' \in \delta S$.
For $S_0'=S_0\setminus\{s_i\}$, let this set be $S_i := S_0\setminus\{s_i\}\cup\{s_i'\}$, where $s_i' \in V(H)$. 
Then the set $\mathcal{S'}$ consists precisely of $S_0$ and $S_i$, for $i \in [r]$.
We claim that this implies that
\begin{description}
\item[($\star\star$)] there is a single vertex $z \in V(H)$ such that $S_i=S_0\setminus\{s_i\}\cup\{z\}$ for any $s_i \in S_0$.
\end{description}
Indeed, let $s_i, s_j \in S_0$ be distinct elements, and consider the set $S''=S_0\setminus\{s_i,s_j\}\cup\{s_i'\} \in \delta S_i$.
By assumption, $S''$ is a subset of an even number of sets in $\mathcal{S'}\cup\{S_0\}$.
On the other hand, $S''$ cannot occur in $\delta S_a$ for any $a \notin \{i,j\}$ since $s_a \in S''$ but $s_a \notin S_a$.
Hence, $S_0\setminus\{s_i,s_j\}\cup\{s_i'\} \in \delta S_j$, which implies $s_i'=s_j'$, and since $i$ and $j$ were arbitrary, this proves ($\star\star$).

Now we define replacement vertices $s_i''$ for $s_i \in S_0$ to construct a lower bound structure of order $r$, which will yield a contradiction.
Let $i \in [r]$ be such that $z \in L_i$, and note that this exists since $z$ is found in $F$.
Then $z$ and $s_i$ are incomparable by assumption, and we may define $s_i''=z$.
For every $j \in [r]$, $j \neq i$, set $s_j''=s_i$; these are incomparable vertices since $S_i$, $S_j$ have common neighbors in $L$ but $S_i\cup\{s_i\}$ and $S_j\cup\{s_j\}$ do not.  
Finally, the set $S_0$ has no common neighbor in $L$ by assumption.
Now consider a set $S=(S_0 \setminus S') \cup \{s'' \mid s \in S'\}$ for a non-empty $S' \subseteq S_0$.
Then, either $S'=\{s_i\}$ and thus $S=S_i$, or there exists some $j \in [r]$, $j \neq i$, such that $s_j \in S'$.
In the latter case $S \subseteq S_j$, and in both cases $S$ has a common neighbor in $L$.
Hence $L$, $(s_1,\ldots,s_r)$, and $(s''_1,\ldots,s''_r)$ form a lower bound structure of order $r=c^*(H)>d^*(H)$, a contradiction.
We conclude that the polynomial $p(\vect{y})$ exists, which completes the proof. 
\end{proof}

Combining \Cref{thm:lower-bound,thm:kernelization-by-polynomials,thm:kernel-delta}, and \Cref{lem:delta-c-plus1,lem:bd-poly}, we conclude the following.

\thmbddeg*


\bibliographystyle{abbrv}
\bibliography{ref} 

\end{document}